\documentclass[11pt]{amsart}

\usepackage[english]{babel}     
\usepackage[utf8]{inputenc}
\usepackage{amsmath}
\usepackage{amsfonts}
\usepackage{amssymb}

\usepackage{mathrsfs}
\usepackage{stmaryrd}
\usepackage{hyperref}
\usepackage{xcolor}
\usepackage{todonotes}
\usepackage{comment}
\usepackage{enumerate}
\usepackage{tikz-cd}
\usepackage[capitalize]{cleveref}
\usepackage{tikz}
\usepackage[T1]{fontenc}
\usepackage{todonotes}

\usetikzlibrary{shapes,calc,arrows,intersections,decorations.pathreplacing,decorations.markings,shapes.geometric,calc,positioning,backgrounds}

\newtheorem{theorem}{Theorem}[section]

\newtheorem{lemma}[theorem]{Lemma}

\newtheorem{corollary}[theorem]{Corollary}

\theoremstyle{definition}
\newtheorem{definition}[theorem]{Definition}

\newtheorem{remark}[theorem]{Remark}

\theoremstyle{remark}
\numberwithin{equation}{section}

\newcommand{\Pol}{\textrm{Pol}}

\newcommand{\bbP}{\mathbb{P}}
\newcommand{\tr}{\mathrm{tr}}

\newcommand{\Link}{{\rm Link}}
\newcommand{\cP}{\mathcal{P}}
\newcommand{\cross}{\textrm{cr}}
\newcommand{\boldi}{{\bf i}}

\newcommand{\id}{\textrm{id}}

\newcommand{\order}{\mathcal{O}}
\newcommand{\minotimes}{\otimes_{{\rm min}}}

\title[Strongly convergent matrix models for $q$-Gaussian algebras]{Strongly convergent matrix models for $q$-Gaussian algebras}

\date{\noindent \today.  MSC2010 keywords:  46L51, 46L54 , 15B52. }

\author[Martijn Caspers]{Martijn Caspers}
\author[Enli Chen]{Enli Chen}

\address{TU Delft, EWI/DIAM,
	P.O.Box 5031,
	2600 GA Delft,
	The Netherlands}

\email{M.P.T.Caspers@tudelft.nl}

\email{E.Chen-1@tudelft.nl}

\begin{document}

\maketitle

\begin{abstract} 
We construct strongly convergent  finite-dimensional random matrix models for finite $q$-Gaussian family in the range   $|q| < \sqrt{2}-1$. The construction has two stages. First, we show that normalized sums of graph-product semicirculars over an Erd\H{o}s--R\'enyi graph   satisfy the $q$-Toeplitz relations up to an operator norm error converging to zero in probability. Using  ultraproduct methods, this yields complete strong convergence, uniformly over all matrix coefficient dimensions, for noncommutative polynomials with bounded degree. Second, we use a quantitative tensor-GUE for graph-product semicirculars to convert these operator models into finite-dimensional random matrices. For every fixed polynomial degree, the convergence is uniform over coefficient dimensions that may be larger than the dimension of the random matrices. As applications, in the above range of $q$, the $C^*$-algebra generated by a finite $q$-Gaussian family is MF, and the Brown–Douglas–Fillmore extension semigroup of the nontrivial $C^*$-algebra generated by a  finite $q$-Gaussian family is not a group.

\end{abstract}
\section{Introduction}
The $q$-Gaussian processes and $q$-semicircular variables, $-1 \leq q \leq 1$, have been introduced and developed by Bo\.{z}ejko and Speicher \cite{BozejkoSpeicher} and Bo\.{z}ejko, K\"ummerer and Speicher \cite{BozejkoKummererSpeicher}. They interpolate between Bosonic classical fields of Gaussian variables ($q=1$) and Fermionic CAR algebras ($q=-1$). The particular parameter $q=0$ gives a free semicircular system and is of fundamental importance in Voiculescu's theory of free probability.  The aim of this paper is to construct random matrix tuples such that norms of noncommutative polynomials in that tuple   converge  to norms of the polynomial evaluated at a  $q$-semicircular system.  The main feature of our construction is that convergence holds even when the coefficient dimension grows beyond the size of the random matrices.

\vspace{0.3cm}
 
Let us give an overview of the origin and motivation of strong convergence results. Firstly, distinction between  weak convergence (convergence in moment) and strong convergence  is essential. Voiculescu proved that the independent GUE (`Gaussian Unitary Ensemble') converges in joint $*$-moments to a free semicircular family \cite{Voi91}. Haagerup and Thorbj{\o}rnsen subsequently proved convergence of the
operator norm of every noncommutative polynomial \cite{HT05}. Together with weak convergence, this is called \emph{strong convergence}; it retains the limiting $C^*$-norm and rules out spectral outliers that are invisible to normalized traces. Subsequent developments include the real and symplectic Gaussian ensembles \cite{Sch05}, deterministic deformations of GUE matrices \cite{CoMa14}, and the Wigner-type ensembles \cite{And13}. 

A substantially stronger question is whether strong convergence remains
uniform under growing matrix coefficients. More precisely, assume an ensemble $X^{(N)}$ to converge to $X$. For
matrices $A_w^{(N)}\in M_{D_N}(\mathbb C)$ one considers
\[
  P_A(X^{(N)})
  =
  \sum_{|w|\leq d}A_w^{(N)}\otimes X_w^{(N)},
\]
where $w$ is a word and $X_{w}^{(N)}$ is the monomial of the ensemble described by $w$.  For every fixed $D_N=D$, this is a finite amplification of ordinary
strong convergence. When $D_N\to\infty$, however, scalar strong
convergence gives no uniform control over the expanding coefficient
ball. Such uniformity probes the higher matrix norms of the
bounded-degree polynomial space and hence its operator-space structure,
rather than only its scalar $C^*$-norm. Equivalently, it is a
finite-scale version of complete norm convergence and it is relevant to
completely bounded maps and minimal tensor products; see
\cite{Pisier2003,PisierSubexp}.

Growing matrix coefficients have played an important role in several
developments in random matrix theory and operator algebras. Pisier
connected strong convergence with slowly growing coefficients to the
theory of subexponential operator spaces
\cite{PisierSubexp}. In the related tensor-product problem for
independent GUE families, Pisier obtained strong convergence when the
second matrix size is $o(N^{1/4})$; Collins--Guionnet--Parraud improved
this to $o(N^{1/3})$ \cite{CollinsGuionnetParraud}. Hayes \cite{HayesPT} showed that
reaching the matrix model dimension regime is powerful enough to have major
von Neumann algebraic consequences: the GUE strong convergence
problem with equal coefficients sizes implies under certain assumptions the Peterson--Thom conjecture \cite{PetersonThom}; this problem was subsequently resolved in
\cite{BelCap,BC23} based on Hayes' method.  Thereafter, the polynomial method of \cite{CGTvH,CGVvH} has provided significant improvement of quantitative strong convergence
estimates with large matrix coefficients for several important random matrix models, and has established strong convergence for tensor-GUE, which was conjectured by \cite{MT26}.  In particular, C.-F. Chen--Garza-Vargas--van Handel proved
strong convergence for the classical Gaussian ensembles with 
matrix coefficients of dimension $\exp(o(N))$ \cite{CGVvH}.
  Further   strong convergence results with matrix coefficients and explicit rate functions, preceding \cite{CGVvH},  are contained in \cite{BBVH23,   Par22, Par23, Par24,MageedelaSalle},    
 and for further background, we refer the reader to the surveys
\cite{vH26a,vH26b} and the references therein.   

\vspace{0.3cm}

\subsection{Strongly convergent random matrix models for $q$-Gaussian families for $|q|<\sqrt{2}-1$}
The main result of the present paper is a strong convergence result with matrix coefficients growing faster than the model dimension  for a finite tuple of $q$-semicircular variables (Theorem \ref{thm:q-gaussian-complete-matrix-model}) in the range $\vert q \vert < \sqrt{2} -1$.

\vspace{0.3cm}

In passing we note that for  
$|q|<q_0(r)$, free monotone transport
\cite{GuionnetShlyakhtenko} represents an $r$-tuple of $q$-Gaussians
by analytic noncommutative functions of a free semicircular family.
Combining this representation with strong GUE convergence gives an
indirect strong approximation in a generator-dependent small-$q$
regime.

\vspace{0.3cm}
 
To approach our main result here several finite-dimensional approximations of $q$-Gaussian families
were previously known at the level of noncommutative distribution. Speicher's `$\epsilon$-free' central limit theorem \cite{Spe92} produces the standard spin-matrix approximation in joint $*$-moments, while \'Sniady \cite{SniadyCMP} constructed Gaussian
random-matrix models for $0<q<1$ interpolating between Gaussian and semicircular families. These results do not by themselves
control the norms of arbitrary noncommutative polynomials.

In this paper we use yet another model that is rooted in the work of M\l otkowski \cite{Mlot} and shows weak convergence of averages of large  sums of semicirculars that randomly are free or commuting, where randomness is dictated by the Erd\H{o}s-R\'enyi random graph with edge probability $q$, in case $0 \leq q < 1$. This model is still infinite dimensional but we approximate the semicirculars in this model by GUE's to obtain finite dimensional models.

\vspace{0.3cm}


To achieve this we need to overcome a difficulty  that is of  a different nature  than prior results: besides
the coefficient dimension, the tensor-interaction given by the  Erd\H{o}s-R\'enyi   graph and the number
of variables also grow. We establish the required uniformity for these
varying tensor-GUE models and ultimately obtain $q$-Gaussian random matrix
models with dimensions $n_k$ whose admissible coefficient dimensions $D_k$ satisfy
\[
  D_k>n_k.
\]
Thus
the approximation retains bounded-degree matrix norms beyond the
matrix model size amplification level.



In fact, C.-F. Chen--Garza-Vargas--van Handel proved that tensor-GUE models associated with a fixed interaction graph converge strongly to the corresponding graph-product semicircular family \cite[Theorem~9.8]{CGVvH}. Their tensor-GUE strong convergence  theorem  concerns a fixed graph and, in its sated form, a fixed polynomial with fixed matrix-coefficient dimension. They also observe that a quantitative version can be obtained by bootstrapping their argument, based on their recent developed powerful polynomial methods \cite{CGTvH, CGVvH}. However, the question for coefficient-uniform estimates remain open. Further, we note that their convergence rates worsen if the size of the graph increases causing a serious difficulty in the analysis.  

Our use of tensor-GUE models differs in two respects. First, the interaction graph is an Erd\H os--R\'enyi graph whose number of vertices tends to infinity, and the normalized averages of the corresponding graph-product semicirculars converge to a $q$-Gaussian family. Second, we require the strong convergence of tensor-GUE are uniform over bounded-degree polynomials whose matrix coefficients and coefficient dimensions may vary with the random matrix size. Establishing this uniformity is a substantial part of the argument.

Building on their quantitative Hayes' model strong convergence results and universality argument, we establish a growing-coefficient refinement for tensor-GUE models, which allow the matrix coefficient dimension to exceed the dimension of the matrix models themselves. 

The construction proceeds in two stages. In the first stage, we show
that normalized sums of graph-product creation operators (when $0\leq q < 1$) or Clifford-twisted graph-product creation operators (when $-1\leq q<0$) indexed by
$\Gamma_k\sim \Gamma(rk,|q|)$, namely $L_{k,a}^{q,\Gamma_k},a=1,\cdots,r$, satisfy the multivariate approximate $q$-Toeplitz
relations, i.e. $q$-Toeplitz
relations up to an operator-norm error (see \cref{Lem=LqToeplitzDifferent}):
\[
 (L_{k,a}^{q,\Gamma_k})^*L_{k,b}^{q,\Gamma_k}
 =\delta_{ab}1+qL_{k,b}^{q,\Gamma_k}(L_{k,a}^{q,\Gamma_k})^*+E_{k,ab}^{q,\Gamma_k},
\]
where, with probability tending to one,
\[
 \max_{a,b}\|E_{k,ab}^{q,\Gamma_k}\|
 =\mathcal O_{\mathbb{P}, q,r}\!\left(\frac{\log k}{\sqrt{k}}\right).
\]
The error estimate follows from the logarithmic size of the 
clique number of an Erd\H os--R\'enyi graph (see Lemma \ref{Lem=CliqueTailBound} or \cite{BollobasErdos}),  as well as  the square-root size of the centered adjacency matrix of an  Erd\H os--R\'enyi graph (see Theorem \ref{Thm=RandomMatrixTailBound}  from \cite{BaiYin}, \cite[Theorem 1.4]{Tro12}).

Conceptually, this relation implements a $q$-deformed non-backtracking word reduction: it separates immediate creation-annihilation contractions from reduced propagation, while the random-graph fluctuation are absorbed into the vanishing error $E_{k,ab}$; see Remark~\ref{rem:nonbacktracking}.

Then using the universal $q$-Toeplitz algebra and an ultraproduct argument, we identify the limiting representation with the $q$-Fock representation. Nuclearity and exactness then upgrade
fixed matrix coefficient strong convergence to complete bounded-matrix coefficient strong convergence. 

One of the key ingredient for the ultraproduct argument is applying the faithfulness of the Fock space representations of the universal $q$-Toeplitz algebras. Although it is expected to hold true for all $-1< q< 1$, it was only proved for range $|q|<\sqrt{2}-1$ (\cite{PuszWoronowicz}). This is where our specific range of $q$ for the later strong convergence of our random matrix models occur. 

In the second stage, we develop a growing-coefficient refinement of
the tensor-GUE argument of \cite{CGVvH}. A dimension-free Fock-space
comparison gives the quantitative estimate (see Lemma \ref{asytompzeroCLTmodulus})
\[
 \sup_{\substack{D\geq1\\ \Lambda(A)\leq M}}
 \left|\|\mathcal P_A(s^{(T)})\|
   -\|\mathcal P_A(x^{\Gamma})\|
 \right|\leq
 \frac{(2^L-2)M}{\sqrt T}.
\]
Combining this estimate with the quantitative Hayes's model strong convergence theorem
and the universality comparison (developed in \cite{BrailovskayaVanHandel2024}) used in \cite{CGVvH}, we obtain that the norms of linear pencils of tensor-GUE models converge uniformly over coefficient dimensions
$D\leq N^{2L}$. Then we apply exactness compression to obtain the corresponding
uniform lower bound estimate, while a finite polynomial spectral detector
and uniform linearization  extend the result from linear pencils to bounded-degree
noncommutative polynomials.

Finally, we combine these two stages by realizing every graph
$\Gamma_k$ through tensor-coordinate sets and then choosing the GUE
dimension sufficiently large. This produces finite-dimensional
random matrices $X_1^{(k)},\ldots,X_r^{(k)}$ such that, for every
fixed degree $d$ (see \cref{thm:q-gaussian-complete-matrix-model}),

\[
\frac{\|P_k(X_1^{(k)},\ldots,X_r^{(k)})\|}{
\|P_k(s_1^{(q)},\ldots,s_r^{(q)})\|
 }\longrightarrow 1
\]
in probability, uniformly for nonzero matrix-valued noncommutative polynomials
$P_k$ of degree at most $d$ and coefficient dimensions up to $D_k$,
where $D_k$ may be chosen larger than the dimension of the random
matrices themselves.

\vspace{0.3cm}

 The construction of matrix models of this paper applies  to  the explicit generator-independent range $|q|<\sqrt{2}-1$, and provides quantitative complete bounded-degree strong convergence at growing matrix-coefficient allowed larger than the dimension of the matrix model. To the best of our knowledge, this is the first finite-dimensional random-matrix models proved to converge strongly to nonfree $q$-Gaussian families throughout the explicit generator-independent range of non-zero $q$ uniform over matrix-coefficient dimensions.

\vspace{0.3cm}

In this context we also mention that very recently Shlyakhtenko \cite{ShlyakhtenkoStrongC} showed that `CAR-UE' matrix models (see also \cite{PisierShl}) do not strongly converge to semicircular variables. Heuristically this model replaces  Gaussian random variables, appearing as matrix entries of  the  GUE's, by CAR algebras. On the other hand in  \cite[Section 4]{ShlyakhtenkoStrongC} it is shown that if these Gaussian variables are replaced by $q$-Gaussians $-1 < q < 1$   strong convergence to the semicircular system holds. Note that in  \cite{ShlyakhtenkoStrongC} the limit model is  semicircular ($q=0$) whereas in the current paper we consider different, finite dimensional,  models that converge to any $q$-Gaussian (at least if $\vert q \vert < \sqrt{2} -1$).

\subsection{Sharp Erd\H os--R\'enyi Graph-Product Degree-One Khintchine Inequality}
As a first consequence of the approximate $q$-Toeplitz relation, we obtain an asymptotically sharp degree-one graph-product Khintchine inequality for Erd\H os--R\'enyi graph products of semicircular variables which for $-1 \leq q < 1$ states that (see \cref{Thm=MainKhintchine}):
\begin{equation}\label{Eqn=KhinIntro}
\left\|\frac1{\sqrt{k}}\sum_{v=1}^{k} x_v^{q,\Gamma_k}  
 \right\|
 \leq
 \frac{2}{\sqrt{1-q}}  + \mathcal O_{\mathbb P}\left(\frac{\log k}{\sqrt{k}}\right),\qquad \left\|\frac1{\sqrt{k}}\sum_{v=1}^k x_v^{q,\Gamma_k}\right\|\longrightarrow  \frac{2}{\sqrt{1-q}}.
 \end{equation}
 in probability, where the semicircular variables $x_v^{q,\Gamma_k}$ are labeled by $\Gamma(k,\vert q \vert)$; in case $-1 \leq q <0$ the variables are Clifford twisted versions of semicirculars.   The limiting constant is optimal, since $2/\sqrt{1-q}$ is precisely the norm of a
standard $q$-Gaussian variable.  This complements the deterministic
graph-product Khintchine inequalities of
\cite{CKL20} and  the recent degree-one graph-product Khintchine inequalites
of \cite{CollinsMiyagawa,SantosEtAl}. These works treat
arbitrary operator coefficients for a fixed graph and bound the norm
in terms of adjacency eigenvalues, clique numbers, or the associated
trace monoid. These results are deterministic and thus hold for all graphs.   For an Erd\H{o}s--R\'enyi graph, applying their general
clique-based graph-product Khintchine inequalities gives, where the first estimate holds everywhere, 
\[
 \left\|\frac1{\sqrt{k}}\sum_{v=1}^{k}x^{\Gamma_k}_v
 \right\|
 \leq 2\sqrt{\omega(\Gamma_k)}
 = \order_{\mathbb P}(\sqrt{\log k}),
\]
whereas our probabilistic argument identifies the finite sharp limit
$2/\sqrt{1-q}$ and provides a quantitative error estimate.  

Note that the Khintchine inequality \eqref{Eqn=KhinIntro} is the simplest case of our strong convergence result where the polynomial is linear and scalar valued. We stress that this specific case is already  new.

\subsection{MF Property of $q$-Gaussian $C^*$-algebras}
 As a  natural corollary from the strong convergent random matrix models for $q$-Gaussian family,  applying a diagonal choice of the parameters of such random matrices gives strong
convergence for all noncommutative polynomials and implies the MF
property of the associated $q$-Gaussian $C^*$-algebra with $|q|<\sqrt{2}-1$. 

Then using the recent proof \cite{AmrutamJekelWasilewski2026} of the simplicity of the $q$-Gaussian $C^*$-algebras, we show that the $q$-Gaussian $C^*$-algebras $A_q(\mathbb C^r)$ with $r\geq 2$ are not quasidiagonal. Hence combining this with  the MF property, by Brown's criterion \cite{Brown2004}, we deduce that the Brown–Douglas–Fillmore extension semigroup $\mathrm{Ext}(A_q(\mathbb C^r))$ is not a group for $r\geq 2$ and $|q|<\sqrt{2}-1$.

\subsection{Paper Overview}
The paper is organized as follows. In \cref{sec:preliminaries} we fix
notation. In \cref{sec:singlestrongconvergence}, we establish the Erd\H os--R\'enyi graph-product operators model approximation for $q$-Gaussians, as well as the sharp degree-one graph product Khintchine inequality for the Erd\H os--R\'enyi graph. In \cref{sec:weakconvergence}, we obtain a quantitative weak convergence for the operator model on the level of $q$-Toeplitz variables. In \cref{sec:multistrongconvergence}, we pass the single variable strong convergence of the operator model to
 complete strong convergence with growing matrix coefficients.  In \cref{sec:strongconvergencematrix}, we convert the
operator models into finite random matrix models via the tensor-GUE models, and prove their strong convergence with growing matrix coefficients.  The applications to the MF property and extension semigroup  are
proved in \cref{sec:MF}.

\subsection{Acknowledgements.} Authors thank Luca Junk, Manish Kumar, Mikael de la Salle and Roland Speicher for helpful discussions at the early stage of this project.

\section{Preliminaries}\label{sec:preliminaries}

 All $C^*$-algebras are complex and unital, all tensor product are minimal, and Hilbert-space innner products are linear in the first variable.
\subsection{Notation, polynmials and strong convergence}\label{subsec:general-notation}

We write $[m]=\{1,\ldots,m\}$, $M_D=M_D(\mathbb C)$, and
\(
   \tr_D(A)=D^{-1}\mathrm{Tr}(A).
\)
The notation $M_{D,N}$ stands for the $D\times N$ complex matrices,
$(M_D)_{\rm sa}$ for the self-adjoint matrices, and $\id_D$ for the
identity map on $M_D$. For a self-adjoint element $a$, its spectrum is
$\mathrm{sp}(a)$. If $K,L\subset\mathbb R$ are nonempty compact
sets, then
\[
 d_{\rm H}(K,L)
 =\max\left\{
   \sup_{x\in K}\mathrm{dist}(x,L),
   \sup_{y\in L}\mathrm{dist}(y,K)
 \right\}
\]
denotes their Hausdorff distance, and
\[
 K+[-\varepsilon,\varepsilon]
 =\{x+t:x\in K,\ |t|\leq\varepsilon\}.
\]

 We write $Z_k=O_{\mathbb P}(b_k)$ if
$Z_k/b_k$ is bounded in probability, and $Z_k=o_{\mathbb P}(b_k)$ if
$Z_k/b_k\to0$ in probability. An event holds \emph{with high
probability} if its probability tends to one. For a family
$(Z_{k,\theta})_{\theta\in I_k}$, uniform convergence in probability
means
\[
  \sup_{\theta\in I_k}
  \mathbb P(|Z_{k,\theta}|>\varepsilon)\longrightarrow0
  \qquad(\forall\varepsilon>0).
\]

We use $\mathbb C\langle x_1,\ldots,x_r\rangle$ for noncommutative polynomials in self-adjoint variables $x_1, \ldots, x_r$.   For a word $w=i_1\cdots i_m$, put
\[
 |w|=m,\qquad
 x_w=x_{i_1}\cdots x_{i_m},\qquad
 x_\varnothing=1.
\]
Every $M_D$-valued polynomial of degree at most $d$ has a unique
expression
\[
  P(x)=\sum_{|w|\leq d}A_w\otimes x_w,
  \qquad
  \Lambda_d(P)=\sum_{|w|\leq d}\|A_w\|.
\]
Its involution is determined by
\(
 (A_w\otimes x_w)^*
 =A_w^*\otimes x_{w^{\rm rev}}.
\)
For $A_0,\ldots,A_V\in (M_D)_{\rm sa}$ and \(x=(x_1,\cdots, x_V)\), we use the linear pencil 
\[
  \mathcal P_A(x)
  =A_0\otimes1+\sum_{v=1}^V A_v\otimes x_v,
  \qquad
  \Lambda(A)=\sum_{v=0}^V\|A_v\|.
\]

If $E\subseteq A$ is an operator space, $M_D(E)$ carries the norm
inherited from $M_D\otimes_{\min}A$. For a linear map $u:E\to F$,
\[
  \|u\|_{\rm cb}
  =\sup_{D\geq1}\|\id_D\otimes u\|.
\]

Let $X^{(N)}$ and $x$ be self-adjoint $r$-tuples in tracial
$C^*$-probability spaces $(A_N,\tau_N)$ and $(A,\tau)$, respectively.
We say that $X^{(N)}$ converges \emph{strongly in probability} to $x$
if, for every scalar noncommutative polynomial $P$,
\[
 \tau_N(P(X^{(N)}))\longrightarrow\tau(P(x)),
 \qquad
 \|P(X^{(N)})\|\longrightarrow\|P(x)\|
\]
in probability \cite{Male2012}.

For a sequence $D_N\geq1$, convergence is \emph{uniform up to
coefficient dimension $D_N$ at bounded degree} if, for every fixed
$d,M$ and $\varepsilon>0$,
\[
 \sup_{\substack{1\leq D\leq D_N,\ \deg P\leq d\\
                  P\in M_D\otimes
                  \mathbb C\langle x_1,\ldots,x_r\rangle,\ 
                  \Lambda_d(P)\leq M}}
 \mathbb P\!\left(
   \big|\|P(X^{(N)})\|-\|P(x)\|\big|>\varepsilon
 \right)
 \longrightarrow0.
\]
If the supremum is over all $D\geq1$, we call this
\emph{complete bounded-degree convergence}.
\subsection{Graphs and Erd\H{o}s--R\'enyi models}
\label{subsec:random-graphs}

A graph $\Gamma=(V\Gamma,E\Gamma)$ is finite, simple, and
undirected. We write $u\sim_\Gamma v$ for adjacency, $A_\Gamma$ for
the zero-diagonal adjacency matrix. A clique is a subgraph whose vertices are pairwise
adjacent, and $\omega(\Gamma)$ denotes the clique number, i.e. the maximal size of cliques in $\Gamma$.

We use reduced graph products in the sense of \cite{CaspersFima}.
A \textit{graph-product semicircular} family $(s_v)_{v\in \Gamma}$ is
the canonical family in the reduced graph product of standard
semicircular probability spaces. Thus adjacent vertex algebras
commute and are tensor independent, while an edgeless graph gives free
independence. The concrete graph-product Fock space
 is introduced in \cref{sec:singlestrongconvergence}.

For $p\in[0,1]$, $\Gamma(k,p)$ denotes the \textit{Erd\H{o}s--R\'enyi graph} on
$[k]$ whose edges occur independently with probability $p$. If we index the vertices of $\Gamma(\infty, p)$ with the natural numbers than we may naturally see $\Gamma(k,p)$ as a random subgraph of $\Gamma(\infty, p)$ by identifying it with the first $k$ vertices. This way we can make sense of limits $k \rightarrow \infty$ that converge almost everywhere or that converge in probability by viewing $\Gamma(k,p)$ as subgraph of  $\Gamma(\infty, p)$. 

We write
$\mathbb P_\Gamma$ and $\mathbb E_\Gamma$ for its law and expectation.

In the multivariable construction we sample one graph $\Gamma_k$ on
\[
   V_k=[k]\times[r]
      =V_{k,1}\mathbin{\dot\cup}\cdots
       \mathbin{\dot\cup}V_{k,r},
   \qquad
   V_{k,a}=[k]\times\{a\},
\]
including all cross-block edges. Throughout the $q$-Toeplitz and  $q$-Gaussian model,
the edge probability is $p=|q|$.

For negative $q$ we also use the Clifford
$C^*$-algebra $\mathrm{Cliff}(V_k)$ generated by self-adjoint Clifford unitaries $(c_v)_{v\in V_k}$ satisfying
\[
 c_uc_v=-c_vc_u\quad(u\neq v),
 \qquad
 c_v^2=1,
\]
equipped with its canonical trace $\tau_{\rm Cl}$, which is characterized by $\tau_{\rm Cl}(1)=1$ and by vanishing on every nontrivial reduced word of Clifford unitaries.

\subsection{The $q$-Fock space and
$q$-Toeplitz algebras}
\label{subsec:q-fock-concise}

Fix $-1<q<1$, let $H=\mathbb C^r$, and let
$e_1,\ldots,e_r$ be its standard basis. On
\[
 \mathcal F_{\rm alg}(H)
 =\mathbb C\Omega_q
  \oplus\bigoplus_{n\geq1}^{\rm alg}H^{\otimes n},
\]
where $\Omega_q$ is the vacuum vector, the $q$-inner product on $H^{\otimes n}$ is
\[
 \langle\xi,\eta\rangle_q
 =\langle P_q^{(n)}\xi,\eta\rangle,
 \qquad
 P_q^{(n)}
 =\sum_{\sigma\in S_n}
 q^{\mathrm{inv}(\sigma)}U_\sigma.
\]
Here $U_\sigma$ denotes the natural permutation of tensor factors.
For $|q|<1$ these forms are strictly positive, and their completion is
the $q$-Fock space $\mathcal F_q(H)$
\cite{BozejkoSpeicher,DykemaNica}.

The left creation operator on $\mathcal F_q(H)$ is
\[
 \ell_q(h)\Omega_q=h,
 \qquad
 \ell_q(h)(\xi_1\otimes\cdots\otimes\xi_n)
 =h\otimes\xi_1\otimes\cdots\otimes\xi_n.
\]
Writing $\ell_{q,a}=\ell_q(e_a)$, we has
\[
 \ell_{q,a}^*\ell_{q,b}
 =\delta_{ab}1+q\ell_{q,b}\ell_{q,a}^*.
\]

Set
\[
 s_a^{(q)}
 =\ell_{q,a}+\ell_{q,a}^*,
 \qquad
 \mathcal T^{\rm red}_{q,r}
 =C^*(\ell_{q,1},\ldots,\ell_{q,r}),
\]
the latter is called the \textit{reduced $q$-Toeplitz algebra}. And write the \textit{$q$-Gaussian $C^*$-algebra} and \textit{$q$-Gaussian von Neumann algebra} respectively as
\[
 A_q(\mathbb C^r)
 =C^*(s_1^{(q)},\ldots,s_r^{(q)}),
 \qquad
 \Gamma_q(\mathbb R^r)
 =W^*(s_1^{(q)},\ldots,s_r^{(q)}).
\]

The vacuum functional
$\varphi_q(x)
 =\langle x\Omega_q,\Omega_q\rangle$,
 $x\in \mathcal T^{\rm red}_{q,r}$
is not tracial on $\mathcal T^{\rm red}_{q,r}$. Its restriction to
$A_q(\mathbb C^r)$ is a faithful trace, denoted by $\tau_q$
\cite[Proposition~2.3]{BozejkoKummererSpeicher}. We have 
\[
 \|s_a^{(q)}\|=\frac{2}{\sqrt{1-q}}\qquad \textrm{and}\qquad
 \tau_q(s_{i_1}^{(q)}\cdots s_{i_m}^{(q)})
 =\sum_{\pi\in\mathcal P_2(m)}
   q^{\mathrm{cr}(\pi)}
   \prod_{\{u,v\}\in\pi}\delta_{i_u,i_v}.
\]
Here $\mathcal P_2(m)$ is the set of pair partitions of $[m]$,
$\mathrm{cr}(\pi)$ is the crossing number of $\pi$, and the sum is
zero for odd $m$ \cite{BozejkoKummererSpeicher}.

Let $\mathcal T_{q,r}$ be the unital $*$-algebra generated by
$T_1,\ldots,T_r$ with \textit{the $q$-Toeplitz relations}
\[
 T_a^*T_b=\delta_{ab}1+qT_bT_a^*,
\]
and let $\mathcal T_{q,r}^u$ be its universal $C^*$-completion. We call $\mathcal T_{q,r}^u$ \textit{the universal $q$-Toeplitz algebra}. Then 
\[
\Phi_q:\mathcal T_{q,r}^u \rightarrow  \mathcal T^{\rm red}_{q,r}
\]
denotes the canononical quotient map.

\subsection{Ultaproduct and approximation techniques}

 For nonprincipal ultrafilter $\mathcal U$ and $C^*$-algebras $(A_k)$, their $C^*$-ultraproduct is 
 \[\prod_{k,\mathcal U}A_k=\ell^{\infty}((A_k))\bigg/\{(a_k):\lim_{k\to\mathcal{U}}\Vert a_k\Vert=0\}, \qquad \|(a_k)_{\mathcal U}\Vert=\lim_{k\to\mathcal U}\Vert a_k\Vert.\]
We use the analogous Hilbert-space ultraproduct notation. We recall the Choi-Effros lifting theorem that we will use: a complete positive contraction from a separable $C^*$-algebra $B$ to a quotient $B/J$ has a complete positive contraction lift to $B$ \cite{ChEf76}.  

The following consequence of exactness is needed in \cref{sec:multistrongconvergence,sec:strongconvergencematrix}.

\begin{lemma}[Fixed-family compression]
\label{lem:fixed-compression}
Let $A$ be a exact $C^*$-algebra and let $E\subseteq A$ be finite dimensional subspace. For
every $\delta>0$ there is $m=m(E,\delta)$ such that, for every
$D\geq1$ and $X\in M_D(E)$, there are contractions
$U,V\in M_{m,D}$ satisfying
\[
 \|(U\otimes1)X(V^*\otimes1)\|
 \geq(1-\delta)\|X\|.
\]
In particular, $m$ is independent of $D$ and $X$.
\end{lemma}

This is \cite[Lemma~5.13]{MageedelaSalle}; it  also follows from the local operator space characterization of exactness \cite[Corollary~17.5]{Pisier2003}.

A separable $C^*$-algebra is \emph{MF} if it embeds into
\[\prod_{k}M_{n_k}\bigg/ \bigoplus_{k} M_{n_k},\qquad \bigoplus_{k} M_{n_k}=\{(a_k):a_k\in M_{n_k},\Vert a_k\Vert\longrightarrow 0\}.  \]
 \subsection{Normalized GUE matrices}\label{subsec:gue}

A normalized $N\times N$ GUE matrix $H^{(N)}$ is a random self-adjoint matrix which has independent  real Gaussian random variables on the 
diagonal entries of variance $N^{-1}$ and independent real and
imaginary parts of Gaussian random variables above the diagonal of variance $(2N)^{-1}$. Thus
\(
 \mathbb E\circ\tr_N\bigl((H^{(N)})^2\bigr)=1,
\)
and its limiting law is the standard semicircular distribution on
$[-2,2]$. All GUE matrices appearing together are independent and
independent of the graph unless stated otherwise; their joint law is
denoted by $\mathbb P_{\Gamma,{\rm GUE}}$.

\section{Strong convergence of $q$-Gaussian models}\label{sec:singlestrongconvergence}

Let $\Gamma$ be a finite graph. Let $\mathcal{F}_\Gamma$ be the graph product Fock space of the graph product of semicircular variables with respect to the vacuum state. 
\[
\mathcal{F}_\Gamma  = \bigoplus_{\mathcal{A}^+_\Gamma}  \mathbb{C} = \ell^2(\mathcal{A}^+_\Gamma), 
\]
where $\mathcal{A}^+_\Gamma$ is the right-angled Artin monoid consisting of all words with letters in $\Gamma$; i.e.  without formal inverses.  For $w \in \mathcal{A}^+_\Gamma$ we let $\delta_w$ be the unit vector  that is 1 in the   summand indexed by $w$ and which is 0 in all other summands. Denote $\Omega_{\Gamma}=\delta_{e}$ as the vacuum vector of $\mathcal F_{\Gamma}$.   

\subsection{Graph Product Model}
For $v \in \Gamma$ we set the left creation operator
\[
\ell_v: \mathcal{F}_\Gamma \rightarrow \mathcal{F}_\Gamma: \delta_w \mapsto \delta_{vw}.
\]
Then $\ell_v^\ast$ is the left annihilation operator. We have the graph product creation relation
\[\ell_w^*\ell_v=\delta_{wv}I+(A_{\Gamma})_{wv}\ell_v\ell_w^*.\]
Then $x^{\Gamma_k}_v=\ell_v+\ell_v^*$, $v\in\Gamma$ are the graph product semicirculars. We also set 
\[
L_\Gamma = \frac{1}{\sqrt{\vert \Gamma \vert }} \sum_{v \in \Gamma}  \ell_v. 
\]
We define the following annihilation-gradient map: 
\[
D_\Gamma: \mathcal{F}_\Gamma \rightarrow  \ell^2 \Gamma \otimes \mathcal{F}_\Gamma\qquad \xi \mapsto \sum_{v \in \Gamma}  \delta_v \otimes \ell^*_v \xi.
\]
It is easy to see that $
D_\Gamma^*(\delta_v\otimes \eta)=\ell_v\eta$, for $\eta\in \mathcal{F}_{\Gamma}$, and  $D_{\Gamma}^*D_{\Gamma}=\sum_{v\in \Gamma}\ell_v\ell^*_v$.
\subsection{Clifford-twisted Graph Product Model}

Let $\{c_v\}_{v\in\Gamma}$ be self-adjoint
Clifford unitaries satisfying
\[
c_v^2=1,\qquad c_uc_v=-c_vc_u
\quad (u\ne v).
\]
Denote the (finite dimensional) von Neumann algebra geneated by them as $\left(\mathrm{Cliff}(\Gamma),\tau_{\mathrm{Cl}}\right)$ where $\tau_{\textrm{Cl}}$ is the usual tracial state   that is 0 on elements of nonzero degree.  
In this case define the Clifford-twisted left creation operator
\[
\widetilde \ell_v:=c_v\otimes \ell_v:L^2(\mathrm{Cliff}(\Gamma),\tau_{\mathrm{Cl}})\otimes\mathcal{F}_{\Gamma}\to L^2(\mathrm{Cliff}(\Gamma),\tau_{\mathrm{Cl}})\otimes\mathcal{F}_{\Gamma}.
\]
Denote $\Omega_{\mathrm{Cl}}$ as the vacuum vector of $L^2(\mathrm{Cliff}(\Gamma),\tau_{\mathrm{Cl}})$, then the vacuum vector of $L^2(\mathrm{Cliff}(\Gamma),\tau_{\mathrm{Cl}})\otimes\mathcal F_{\Gamma}$ is $\Omega_{\mathrm{Cl}}\otimes\Omega_{\Gamma}$.
Then $\widetilde{\ell}_v^\ast$ is the left annihilation operator. We have the Clifford-twisted graph product creation relation
\[\widetilde{\ell}_w^*\widetilde{\ell}_v=\delta_{wv}I-(A_{\Gamma})_{wv}\widetilde{\ell}_v\widetilde{\ell}_w^*.\]
Then $x^{\Gamma,\rm{Cl}}_v=\widetilde{\ell}_v+\widetilde{\ell}_v^*$, $v\in\Gamma$ are the Clifford-twisted graph product semicirculars. Let
\[
\widetilde{L}_\Gamma=\frac1{\sqrt{|\Gamma|}}\sum_{v\in\Gamma} \widetilde \ell_v.
\]
We define the following annihilation-gradient map: 
\[
\widetilde{D}_\Gamma: L^2(\mathrm{Cliff}(\Gamma),\tau_{\mathrm{Cl}})\otimes\mathcal{F}_\Gamma \rightarrow  \ell^2 \Gamma \otimes L^2(\mathrm{Cliff}(\Gamma),\tau_{\mathrm{Cl}})\otimes \mathcal{F}_\Gamma\qquad \xi \mapsto \sum_{v \in \Gamma}  \delta_v \otimes \widetilde{\ell}^*_v \xi. 
\]
Again,
$\widetilde{D}_\Gamma^*(\delta_v\otimes \eta)=\widetilde{\ell}_v\eta$, for $\eta\in \mathcal{F}_{\Gamma}$, and   $\widetilde{D}_{\Gamma}^*\widetilde{D}_{\Gamma}=\sum_{v\in \Gamma}\widetilde{\ell}_v\widetilde{\ell}^*_v=\sum_{v\in\Gamma}\ell_v\ell_v^*$.
\subsection{Approximate q-Toeplitz relation with the graph product and Clifford-twisted graph product model}
Consider the adjacency matrix as operator
\[
A_\Gamma: \ell^2 \Gamma \rightarrow \ell^2\Gamma: \delta_v \mapsto \sum_{w \in \Gamma}  (A_\Gamma)_{wv} \delta_w.
\]
For $-1<q<1$, put
\[
 \varepsilon_q=
 \begin{cases}
   1,  & q\geq0,\\
  -1,  & q<0.
 \end{cases}
\]
For $\Gamma\sim\Gamma(k,|q|)$ define
\[
A_\Gamma^{(q)}=\varepsilon_qA_\Gamma,\qquad B_\Gamma^{(q)}
=A_\Gamma^{(q)}-q(J-I).
\]
Here $J$ is the all-ones matrix, and $I$ is the identity matirx. 

\begin{remark}\label{Rmk=MeanZeroEtc}
  $B_{\Gamma}^{(q)}$ is symmetric matrix. The off-diagonal entries of $B_\Gamma^{(q)}$ are centered and have
variance $|q|(1-|q|)$. In particular, when $q\geq 0$, every off-diagonal entry of $B_{\Gamma}^q$ is mean-zero Bernoulli random variable. 
  
\end{remark}

 We now have the following.

\begin{lemma}\label{Lem=LqToeplitz}
For $\Gamma$ a finite (deterministic) graph   with $k$ vertices and $-1 \leq q \leq 1$ we have 
\begin{equation}\label{Eqn=LqToeplitz} 
   L_\Gamma^\ast L_\Gamma =  q L_\Gamma L_\Gamma^\ast  + I  - \frac{q}{k}  \sum_{v \in \Gamma} \ell_v \ell_v^\ast  +  D_\Gamma^\ast (   \frac{ B_\Gamma}{k} \otimes I ) D_\Gamma. 
\end{equation}
and 
\begin{equation}\label{Eqn=QToeplitz}
\widetilde{L}_\Gamma^\ast \widetilde{L}_\Gamma =  q \widetilde{L}_\Gamma \widetilde{L}_\Gamma^\ast  + I  - \frac{q}{k}  \sum_{v \in \Gamma} \widetilde{\ell}_v \widetilde{\ell}_v^\ast  +\widetilde{D}_\Gamma^\ast (   \frac{ \widetilde{B}_\Gamma}{k} \otimes I ) \widetilde{D}_\Gamma. 
\end{equation}
\end{lemma}
\begin{proof} Proof of \eqref{Eqn=LqToeplitz}.  
The graph-product creation operators satisfy
\[
\ell_u^*\ell_v=
\delta_{uv}I+
(A_\Gamma)_{uv}\ell_v\ell_u^*.
\]
Therefore
\[
\begin{aligned}
L_\Gamma^*L_\Gamma&=
\frac1k
\sum_{u,v\in \Gamma}
\ell_u^*\ell_v=
\frac1k\sum_{u,v\in \Gamma}
\left(\delta_{uv}I+(A_\Gamma)_{uv}\ell_v\ell_u^*\right)    =I+\frac1k\sum_{u,v\in \Gamma}(A_\Gamma)_{uv}\ell_v\ell_u^* .
\end{aligned}
\]
Since $A_\Gamma$ has zero diagonal, we can write this identity as
\[
L_\Gamma^*L_\Gamma=I+\frac1k
\sum_{\substack{u,v\in \Gamma\\ u\ne v}}(A_\Gamma)_{uv}\ell_v\ell_u^* .
\]
Next, we have
\[
L_\Gamma L_\Gamma^*=
\frac1k\sum_{u,v\in \Gamma}\ell_v\ell_u^*.
\]
Hence
\[
qL_\Gamma L_\Gamma^*-
\frac qk
\sum_{v\in \Gamma}
\ell_v\ell_v^*=\frac qk
\sum_{\substack{u,v\in V\Gamma u\ne v}}
\ell_v\ell_u^*.
\]
Since $B_\Gamma=A_\Gamma-q(J-I)$,
 for $u\ne v$,we have
\[
(B_\Gamma)_{uv}=
(A_\Gamma)_{uv}-q\quad\textrm{and }\quad
(B_\Gamma)_{vv}=0.
\]
We now compute the $D_\Gamma$-term. By $D_\Gamma\xi=
\sum_{u\in \Gamma}
\delta_u\otimes \ell_u^*\xi$,
and
$D_\Gamma^*(\delta_v\otimes \eta)=
\ell_v\eta$,
we have
\[
D_\Gamma^*
\left(\frac{B_\Gamma}{k}\otimes I\right)
D_\Gamma=D_\Gamma^*
\left(\frac1k\sum_{u,v\in \Gamma}
(B_\Gamma)_{vu}\,
\delta_v\otimes \ell_u^*
\right)=\frac1k\sum_{u,v\in \Gamma}
(B_\Gamma)_{vu}\ell_v\ell_u^*.
\]
Using $(B_\Gamma)_{vv}=0$, we obtain
\[
D_\Gamma^*\left(
\frac{B_\Gamma}{k}\otimes I\right)
D_\Gamma=\frac1k
\sum_{\substack{u,v\in \Gamma\\ u\ne v}}
\bigl((A_\Gamma)_{uv}-q\bigr)
\ell_v\ell_u^*.
\]
Consequently,
\[
\begin{aligned}
qL_\Gamma L_\Gamma^*-\frac qk
\sum_{v\in \Gamma}\ell_v\ell_v^*+
D_\Gamma^*\left(\frac{B_\Gamma}{k}\otimes I\right)
D_\Gamma&=
\frac qk\sum_{\substack{u,v\in \Gamma\\u\ne v}}\ell_v\ell_u^*+
\frac1k\sum_{\substack{u,v\in \Gamma\\u\ne v}}   \bigl((A_\Gamma)_{uv}-q\bigr)
\ell_v\ell_u^*                          \\&=
\frac1k\sum_{\substack{u,v\in \Gamma\\u\ne v}}
(A_\Gamma)_{uv}\ell_v\ell_u^*.
\end{aligned}
\]
Adding the identity term $I$ we get
\[
L_\Gamma^*L_\Gamma=qL_\Gamma L_\Gamma^*+I-
\frac qk
\sum_{v\in \Gamma}
\ell_v\ell_v^*+
D_\Gamma^*\left(
\frac{B_\Gamma}{k}\otimes I\right)D_\Gamma,
\]
which gives the first identity.

\vspace{0.3cm}

Proof of \eqref{Eqn=QToeplitz}. 
For Clifford-twisted graph product model we have
\[
\widetilde{\ell}_u\widetilde{\ell}_v^*=\delta_{uv}I-(A_{\Gamma})_{uv}\widetilde{\ell}_v\widetilde{\ell}_u^*.
\]
Thus the same proof works with the Clifford-twisted adjacency matrix $-A_{\Gamma}$
in place of $A_\Gamma$. Equivalently, we just replace $B_{\Gamma}$ by $\widetilde{B}_{\Gamma}$,
then we get
\[
\widetilde L_\Gamma^*\widetilde L_\Gamma=
q\widetilde L_\Gamma\widetilde L_\Gamma^*+I-\frac qk
\sum_{v\in \Gamma}
\widetilde\ell_v\widetilde\ell_v^*
+\widetilde D_\Gamma^*
\left(
\frac{\widetilde{B}_\Gamma}{k}\otimes I
\right)
\widetilde D_\Gamma,
\]
which gives the second identity.

\end{proof}

Next we turn to the probabilistic setting. We first make the following remark.

\begin{remark}\label{Rmk=WhyThisWorks}
Assume $0 \leq q < 1$. 
We interpret the terms in \eqref{Eqn=LqToeplitz}   as follows:
\[
\overbrace{ L_\Gamma^\ast L_\Gamma =  q L_\Gamma L_\Gamma^\ast  + I }^{q-\textrm{Toeplitz relation}} -\overbrace{ \frac{q}{k}  \sum_{v \in \Gamma} \ell_v \ell_v^\ast}^{\textrm{remainder 1}}  + \overbrace{ D_\Gamma^\ast (   \frac{B_\Gamma}{k} \otimes 1 ) D_\Gamma}^{\textrm{remainder 2}}. 
\]
We will show in the remainder of this section  that if $\Gamma = \Gamma(k,q)$ and we let $k \rightarrow \infty$ then the two remainder terms will go to 0 with high probability and the $q$-Toeplitz relation remains. This is essentially why our proof works. The analogous observation holds for   \eqref{Eqn=QToeplitz} for $-1 < q < 0$.
\end{remark}

We first consider the necessary estimates that control remainder 1 of Remark \ref{Rmk=WhyThisWorks}. 

\begin{remark}[Starting clique]
 For $w \in \mathcal A^+_\Gamma$ there is a uniquely defined largest clique $\mathrm{Init}_{\Gamma}(w)$ such that $\vert \mathrm{Init}_{\Gamma}(w) w \vert = \vert w \vert - \vert \mathrm{Init}_{\Gamma}(w) \vert$, i.e. $\mathrm{Init}_{\Gamma}(w)$ is the largest clique at the start of $w$. 
\end{remark}

The following lemma is now straightforward and we omit the proof. 

\begin{lemma}\label{numberoperator}
We have for $w \in \mathcal{A}_\Gamma^+$ that $\sum_{v \in \Gamma} \ell_v \ell_v^\ast  \:  \delta_w = \vert \mathrm{Init}_{\Gamma}(w)\vert \delta_w$.  In particular, 
\[
\bigg\Vert \sum_{v \in \Gamma} \ell_v \ell_v^\ast \bigg\Vert = \bigg\Vert \sum_{v \in \Gamma} \widetilde{\ell}_v \widetilde{\ell}_v^\ast \bigg\Vert=\omega(\Gamma).
\]
\end{lemma} 

In case of the random graph $\Gamma = \Gamma(k, q)$  we have good control over the maximal clique size $\omega(\Gamma(k,q))$.  In fact, it is a classical result by Bollob\'as and Erd\H{o}s  \cite[Corollary 1]{BollobasErdos} that  $\omega(\Gamma(k,q))$ is highly concentrated around only two neighbouring natural numbers that are of $\order(\log(k))$.  Here we use a slightly different statement about the tail bound that can be found in \cite[Section 10.2, p. 184]{AlSp}. Its proof is only stated for  $q = \frac{1}{2}$ but holds verbatim for other values of $q$ as is also mentioned in \cite{AlSp}. We give the proof for convenience of the reader.

\begin{lemma}\label{Lem=CliqueTailBound}
Let $0<q<1$. For every constant
$C_q>\frac{2}{\log(1/q)}$,
there exists a constant $c_q>0$ such that
\[
\bbP\!\left(\omega(\Gamma(k,q))\ge C_q\log k\right)
\le
\exp\!\left(-c_q \log^2 k\right)
\]
for all sufficiently large $k$.
\end{lemma}

\begin{proof}
  Let $m=\lceil C_q \log k \rceil$.
If $\omega(\Gamma(k,q))\ge m$, then the graph contains an $m$-clique. Hence, by   the Markov’s inequality,
\[
\mathbb{P}(\omega(\Gamma(k,q))\ge m)
\le\mathbb{E}[X_m],
\]
where $X_m$ denotes the number of $m$-cliques in $\Gamma(k,q)$.

Since every $m$-subset of vertices forms a clique with probability
$q^{\binom{m}{2}}$, we have
\[
\mathbb{E}[X_m]
=
\binom{k}{m}
q^{\binom{m}{2}}.
\]
Using the elementary estimate $\binom{k}{m}
\le\left(\frac{ek}{m}\right)^m$,
we obtain
\[
\mathbb{P}(\omega(\Gamma(k,q))\ge m)
\le\left(\frac{ek}{m}\right)^m
q^{m(m-1)/2}.
\]
Taking logarithms gives
\[
\log\mathbb{P}(\omega(\Gamma(k,q))\ge m)
\le
m\log\!\left(\frac{ek}{m}\right)
-\frac{\log(1/q)}{2}m(m-1).
\]
Since $m=C_q\log k+\mathcal O(1)$,
we have
\[
m\log\!\left(\frac{ek}{m}\right)
=C_q\log^2 k+\mathcal{O}(\log k\log\log k),
\]
while
\[
\frac{\log(1/q)}{2}m(m-1)
=
\frac{C_q^2\log(1/q)}{2}\log^2 k
+\mathcal O(\log k).
\]
Therefore
\[
\log\mathbb{P}(\omega(\Gamma (k,q))\ge m)
\le\left(C_q-\frac{C_q^2\log^2(1/q)}{2}\right)
\log^2 k
+    \order(\log k \: \log \log k).
\]
Since
$C_q>\frac{2}{\log(1/q)}$,
the coefficient $C_q-\frac{C_q^2\log(1/q)}{2}$
is strictly negative. Hence there exists $c_q>0$ such that
\[
\log\mathbb{P}(\omega(\Gamma(k,q))\ge m)
\le-c_q\log^2 k
\]
for all sufficiently large $k$. Exponentiating both sides yields
\[
\mathbb{P}\!\left(\omega(\Gamma(k,q))\ge C_q\log k\right)
\le\exp\!\left(-c_q\log^2 k\right).
\]

\end{proof}

To control remainder 2 in Remark \ref{Rmk=WhyThisWorks} we recall the following. Note that by Remark \ref{Rmk=MeanZeroEtc} we are indeed in the setting of the cited references.

\begin{theorem}[Theorem A of \cite{BaiYin} and  Theorem 1.4 of \cite{Tro12}]\label{Thm=RandomMatrixTailBound} 
For $0\leq q\leq 1$, we have 
\[
\lim_{k \rightarrow \infty }k^{-\frac{1}{2}}\Vert B_{\Gamma(k,q)} \Vert =  2 \sqrt{q (1-q)}.
\] 
in probability. Further, there exists $b_q>0$ and $T>0$ such that for every $t > T$ and $k \in \mathbb{N}_{\geq 1}$ we have 
\[
   \mathbb{P}\left( \Vert B_{\Gamma(k,q)}  \Vert   >  t \right) \leq  2k \exp\left( -\frac{b_qt^2}{k + t} \right).
\] 
\end{theorem}
\begin{proof}
    The convergence 
\[k^{-\frac{1}{2}}|| B_{\Gamma(k,q)}||\to 2\sqrt{q(1-q)}\]
 follows from Bai–Yin's result \cite{BaiYin} applied to the centered Bernoulli Wigner matrix $B_{\Gamma(k,q)}$. The non-asymptotic tail estimate 
 \[
\bbP(\|B_{\Gamma(k,q)}\|>t)
\le 2k\exp\!\left(-\frac{b_qt^2}{k+t}\right)
\]
 follows from Tropp’s matrix Bernstein inequality \cite{Tro12}; alternatively see Oliveira's result \cite{Oli10} for adjacency-matrix concentration in random graphs with independent edges.
\end{proof}
\begin{corollary}\label{wignertail}
Let $\alpha \in ( \frac{1}{2} , 1) $. For any $0\leq q\leq 1$, there exists a constant $b'_q>0$ such that for all $k \in \mathbb{N}_{\geq 1}$ we have, 
\[
  \mathbb{P}\left( \Vert B_{\Gamma(k,q)}  \Vert   >  k^{\alpha} \right) \leq  k \exp\left( -b_q' k^{2 \alpha - 1}   \right).
  \]
\end{corollary} 
Now,when $0\leq q\leq 1$, we denote \[L^{(q)}_k:=L_{\Gamma(k,q)}\quad\textrm{and} \quad D_k^{(q)}:=D_{\Gamma(k,q)};\] 
when $-1\leq q<0$, we denote 
\[L^{(q)}_k:=\widetilde{L}_{\Gamma(k,-q)}\quad \textrm{and}\quad D_k^{(q)}:=\widetilde{D}_{\Gamma(k,-q)}.\]  

Combining \cref{Lem=CliqueTailBound,wignertail}, we are now able to estimate $L_{\Gamma(k,q)}$ with high probability.  
 
\begin{theorem}\label{Thm=MainEstimate}  
For every $-1 \leq  q < 1$ there exists $C_q > 0$ such that for every  $\alpha \in (\frac{1}{2}, 1)$ and  $k \in \mathbb{N}$ we have:
\begin{enumerate}
\item\label{Item=LkTail1} for nonnegative $q$:
\[
\mathbb{P}\left( \Vert L^{(q)}_k \Vert^2 \le  \frac{1}{1-q} + C_q \frac{\log k + k^{\alpha}\log k    }{k} \right) \geq 1-\exp\!\left(-c_q \log^2 k\right) -
k \exp\big(-b'_{|q|} k^{2\alpha-1}\big),
\]
\item \label{Item=LkTail2} for negative $q$:
\[
\mathbb{P}\left( \Vert L^{(q)}_k \Vert^2 \le  1 + C_q \frac{\log k + k^{\alpha}\log k    }{k} \right) \geq 1-\exp\!\left(-c_q \log^2 k\right)-
k \exp\big(-b'_{|q|} k^{2\alpha-1}\big).
\]
\end{enumerate}
\end{theorem}
\begin{proof}
By \cref{numberoperator}, we have for every $-1\leq q\leq 1$, $\|D_k^{(q)}\|^2=\omega(\Gamma(k,|q|))$.
Hence
\[
\left\|(D_k^{(q)})^*\left(
\frac{B_k^{(q)}}{k}\otimes1
\right)D_k^{(q)}
\right\|\le\frac{\omega(\Gamma(k,|q|))}{k}\,
\|B_k^{(q)}\|.
\]
We now estimate $\|L_k^{(q)}\|^2$.

If $0\le q<1$, then by Lemma \ref{Lem=LqToeplitz},
\[
\|L_k^{(q)}\|^2
=\|(L_k^{(q)})^*L_k^{(q)}\|  \le
1+q\|L_k^{(q)}\|^2+
\frac{q}{k}\omega(\Gamma(k,q))+
\frac{\omega(\Gamma(k,q))}{k}\|B_k^{(q)}\|.
\]
Therefore
\[
(1-q)\|L_k^{(q)}\|^2\le 1+
\frac{q}{k}\omega(\Gamma(k,q))+
\frac{\omega(\Gamma(k,q))}{k}\|B_k^{(q)}\|.
\]
So for $C_q = 3q/(1-q)$
\begin{equation}\label{estimateLkq}
    \|L_k^{(q)}\|^2\le
\frac{1}{1-q}+\frac{C_q}{3}\left(
\frac{\omega(\Gamma(k,|q|))}{k}+
\frac{\omega(\Gamma(k,|q|))\|B_k^{(q)}\|}{k}\right).
\end{equation}

If $-1 \leq q<0$, then the term $qL_k^{(q)} (L_k^{(q)})^*$ is negative, so in
operator order it can be discarded for an upper bound. Therefore by Lemma \ref{Lem=LqToeplitz},
\begin{equation}\label{Eqn=estimateLq}
\|L_k^{(q)}\|^2\le1+
\frac{|q|}{k}\omega(\Gamma(k,\vert q \vert))+
\frac{\omega(\Gamma(k, \vert q \vert))}{k}\|B_k^{(q)}\|.
\end{equation}

We shall now derive \eqref{Item=LkTail1} from \eqref{estimateLkq}; the derivation of \eqref{Item=LkTail2} from \eqref{Eqn=estimateLq} follows in exactly the same manner.  
Let
\[
\mathcal G_k:=\left\{
\omega(\Gamma(k,|q|))\le 3 \log_{\frac{1}{|q|}} k\right\}\cap\left\{
\|B_k^{(q)}\|\le k^\alpha\right\}.
\]
On $\mathcal G_k$, by \cref{estimateLkq} we get
\[
\|L_\Gamma\|^2\le\frac{1}{1-q}+C_q\left(\frac{  \log_{\frac{1}{|q|}} k}{k}+
\frac{k^\alpha  \log_{\frac{1}{|q|} } k}{k}\right).
\]
We now bound the probability of $\mathcal G_k$. First, by \cref{Lem=CliqueTailBound}, we have
\[
\mathbb P\bigl(
\omega(\Gamma(k,|q|))> 3 \log_{\frac{1}{|q|}} k\bigr)\le \exp\!\left(-c_q \log^2 k\right). 
\]
Second, by \cref{wignertail},
we get for some constant $b_{\vert q \vert}'>0$,
\[
\mathbb P\bigl(
\|B_\Gamma^{(q)}\|>k^\alpha\bigr)\le
k \exp\bigl(-b'_{|q|} k^{2\alpha-1}\bigr).
\]
By the union bound,
\[
\mathbb P(\mathcal G_k)\ge
1-\exp\!\left(-c_q \log^2 k\right)-
k \exp\bigl(-b'_{|q|} k^{2\alpha-1}\bigr).
\]
This proves the theorem.


\end{proof}

Theorem \ref{Thm=MainKhintchine} below is one of the main theorems of this paper and shows that with high probability we have a Khintchine inequality for sums of semi-circular variables that randomly are either free or commutative. Moreover the inequality is sharp by Theorem \ref{Thm=StrongConvergence} we prove below applied to the 1-variable case with a scalar valued linear polynomial and the fact that a $q$-Gaussian generator $s_1^q$ has norm  $2/\sqrt{1-q}$. 

The deterministic analogue of this problem has been addressed in \cite{CollinsMiyagawa, SantosEtAl}; however a bandwidth between the upper and lower bound remained.

\begin{lemma}\label{lemoperatorconverge}
    Let $-1\leq q<1$ and $T\in\mathcal{B}(\mathcal{H})$.  If $r=\Vert T^*T-I-q TT^*\Vert$,
     then 
    we have \[\Vert T+T^*\Vert\leq \frac{2}{\sqrt{1-q}}\sqrt{1+r}.\]
\end{lemma}
\begin{proof}
Denote $R=T^*T-I-qT T^*$, then $r=\Vert R\Vert$.

    When $0\leq q<1$, we have 
    \begin{align*}
       \Vert T\Vert^2=\Vert T^*T\Vert &\leq \Vert I\Vert +q\Vert TT^*\Vert+\Vert T^*T-I-qTT^*\Vert\\&=1+q\Vert T\Vert^2+\Vert R\Vert .  
    \end{align*}
    This give that
    \[\Vert T\Vert\leq \sqrt{\frac{1+ r}{1-q}}.\]
    Therefore \[\Vert T+T^*\Vert\leq 2\Vert T\Vert \leq \frac{2}{\sqrt{1-q}}\sqrt{1+ r}.\]

    When $-1\le q<0$, let $p=-q$. Take a unit vector $\xi\in\mathcal{H}$, we have 
    \[\Vert T\xi\Vert^2+p\Vert T^*\xi\Vert^2=1+\langle R\xi,\xi\rangle.\]
    But $|\langle T\xi,\xi\rangle|\leq \Vert T\xi\Vert$ and $|\langle T\xi,\xi\rangle|=|\langle \xi,T^*\xi\rangle|\leq \Vert T^* \xi\Vert$, we get 
    \[|\langle T\xi,\xi\rangle|^2\leq \min\{\Vert T \xi\Vert^2,\Vert T^* \xi\Vert^2\}.\]
    On the other hand, we have for $a,b\geq 0$, $\min\{a,b\}\leq \frac{a+pb}{1+p}$, therefore
    \[|\langle T\xi,\xi\rangle|^2\leq \frac{\Vert T \xi\Vert^2+p\Vert T^* \xi\Vert^2}{1+p}=\frac{1+\langle R\xi,\xi\rangle}{1+p}.\]
    But \[|\langle (T+T^*)\xi,\xi\rangle|=2 |\Re\langle T\xi,\xi\rangle|\leq 2|\langle T\xi,\xi\rangle|\leq 2\sqrt{\frac{1+\langle R\xi,\xi\rangle}{1+p}}, \]
    we finally get
    \[\Vert T+T^*\Vert=\sup_{\Vert\xi\Vert=1} |\langle (T+T^*)\xi,\xi\rangle|\le \frac{2}{\sqrt{1+p}}\sqrt{1+r}.\]
\end{proof}
\begin{theorem}[Sharp Erd\H os--R\'enyi graph-product degree-one Khintchine inequality]\label{Thm=MainKhintchine}
Let $-1<q<1$.
Then for every $\alpha\in(\frac12,1)$ there exist constants $C_q,c_q>0$ such that for every $k \in \mathbb{N}$
\[
\bbP\left(\left\|
\frac1{\sqrt{k}}
\sum_{v\in\Gamma(k,|q|)}x^{q,\Gamma_k}_v\right\|
\le\frac{2}{\sqrt{1-q}}+
C_q k^{\alpha-1}\log k\right)
\ge 1- \exp\!\left(-c_q \log^2 k\right)
-2k \exp({-c_qk^{2\alpha-1}}).
\]
Consequently,  for fixed $\varepsilon>0$ , there is constant $c_{q,\varepsilon}>0$ such that for large $k$  
\[
\bbP\left(\left\|
\frac1{\sqrt{k}}
\sum_{v\in\Gamma(k,|q|)}x^{q,\Gamma_k}_v
\right\|\le
\frac{2}{\sqrt{1-q}}+\varepsilon
\right)\ge1-\exp{(-  c_{q, \varepsilon}  \log^2 k)}-2k
\exp\left(-c_{q,\varepsilon}
\frac{k}{(\log k)^2}\right),
\]
where 
\[
x^{q,\Gamma_k}_v=
\begin{cases}
x^{\Gamma_k}_v,& q\ge0,\\
x^{\Gamma_k,\mathrm{Cl}}_v,& q<0.
\end{cases}
\]
More sharp,  but without tail bound, we have in probability
\[\left\|
\frac1{\sqrt{k}}
\sum_{v\in\Gamma(k,|q|)}x^{q,\Gamma_k}_v
\right\|\le
\frac{2}{\sqrt{1-q}}+\mathcal{O}(\frac{\log k}{\sqrt{k}}).\]

\end{theorem}
\begin{proof}
 By Lemma \ref{Lem=LqToeplitz}
\[
\bigl(L_k^{(q)}\bigr)^*
L_k^{(q)}=I+
qL_k^{(q)}\bigl(L_k^{(q)}\bigr)^*
+E^{(q)}_k,\]
where
\[E^{(q)}_k
=-\frac{q}{k}\sum_v
\ell_v^{(q)}\bigl(\ell_v^{(q)}\bigr)^*+(D^{(q)}_k)^*
\left(
\frac{B^{(q)}_\Gamma}{k}\otimes I
\right)D^{(q)}_k .
\]
Since
\[(D^{(q)}_k)^*D^{(q)}_k=
\sum_v\ell_v^{(q)}
\bigl(\ell_v^{(q)}\bigr)^*,
\]
and 
\[
\|D^{(q)}_k\|^2=
\omega( \Gamma(k, \vert q \vert)),
\]
we have
\[
\|E^{(q)}_k\|\le
\frac{|q|}{k}\omega(\Gamma(k,|q|))+
\frac{\omega(\Gamma(k,|q|))}{k}
\|B_k^{(q)}\|.
\]
Again, define the good event
\[
\mathcal G_k=
\left\{\omega(\Gamma(k,|q|))\le 3 \log_{\frac{1}{|q|}} k
\right\}\cap\left\{
\|B^{(q)}_k\|\le k^\alpha
\right\}.
\]
Then on $\mathcal G_k$,
\[\|E_k\|\le C_q
\frac{\log k+\log(k)k^\alpha}{k}
=\mathcal{O}(k^{\alpha-1}\log k).
\]
Applying \cref{lemoperatorconverge} with $T_k=L_k^{(q)}$ gives
\[\left\|
\frac1{\sqrt{k}}
\sum_{v\in\Gamma(k,|q|)}x^{q,\Gamma_k}_v
\right\|=\|L_k^{(q)}+(L_k^{(q)})^{*}\|
\le\frac{2}{\sqrt{1-q}}
+  \order(k^{\alpha-1}\log k)
\]
on $\mathcal G_k$.
  Finally,
\cref{Lem=CliqueTailBound} yields 
\[
\bbP\left(
\omega(\Gamma(k,|q|))> 3 \log_{\frac{1}{|q|}} k\right)
\le  \exp({-c_q  \log^2 k}),
\]
while \cref{wignertail} gives
\[
\bbP\left(
\|B_k^{(q)}\|>k^\alpha\right)
\le 2k \exp({-c_qk^{2\alpha-1}}).
\]
And then take the union bound to   completes the proof. 
\end{proof}

\subsection{Multivariable approximate $q$-Toeplitz relation}\label{Sect=Multivariate} 

Now we turn to the $r$-variable case.  We recall the notation for graphs in multivariable case in \cref{subsec:random-graphs}. Now  we sample the Erd\H os--R\'enyi graph on $V_k=[k]\times[r]$, and denote this graph as $\Gamma_k$. For $a\in[r]$, let $\Gamma_{k,a}$ on $V_{k,a}=[k]\times \{a\} $ be the subgraph of $\Gamma_k$ inheriting all edges from $\Gamma_k$.
 Let $\mathcal{A}^+_{\Gamma_k}(\Gamma_{k,a})$ be the words in $\mathcal{A}^+_{\Gamma_k}$ that do not start with a letter in $\mathcal{A}^+_{\Gamma_{k,a}}$. For $w \in \mathcal{A}^+_{\Gamma_k}(\Gamma_{k,a})$ the map $\delta_v \mapsto \delta_{uv}$ identifies $\mathcal{F}_{\Gamma_{k,a}}$ isometrically as a subspace of $\mathcal{F}_{\Gamma_k}$. Moreover, this map intertwines the actions of $\ell_v, v \in \Gamma_{k,a}$ and $L_{\Gamma_{k,a}}$. Hence,  
\begin{equation}\label{Eqn=GraphProductHilbertId}
\mathcal{F}_{\Gamma_k} = \bigoplus_{w \in \mathcal{A}^+_{\Gamma_k}(\Gamma_{k,a}) }  \mathcal{F}_{\Gamma_{k,a} w} \simeq \bigoplus_{w \in \mathcal{A}^+_{\Gamma_k}(\Gamma_{k,a}) }  \mathcal{F}_{\Gamma_{k,a}}. 
\end{equation}
It follows in particular that \ref{Lem=LqToeplitz} holds for $L_{\Gamma_{k,a}}$ acting on $\mathcal{F}_{\Gamma_k}$ through \eqref{Eqn=GraphProductHilbertId}.

Define
\[
  L_{k,a}^{q,\Gamma_k}
  =
  \frac1{\sqrt{k}}
  \sum_{v\in \Gamma_{k,a}}
  \ell_{(a,i)}^{q,\Gamma_k},
\qquad
\textrm{and}\qquad
  S_{k,a}^{q,\Gamma_k}
  =
  L_{k,a}^{q,\Gamma_k}
  +(L_{k,a}^{q,\Gamma_k})^*.
\]

We shall now also show that this $q$-Toeplitz type  relation holds for different generators. 
Sometime we moit the superscripts $\Gamma_k$ in the above notations.

For $a,b\in[r]$, let $B^{(q)}_{k,ab}$ be the block of
$B_{\Gamma_k}^{(q)}$ with rows indexed by $\Gamma_{k,a}$ and columns
indexed by $\Gamma_{k,b}$. Define
\[
   D^{(q)}_{k,a}\xi
  =
  \sum_{u\in\Gamma_{k,a}}
  \delta_u\otimes
  (\ell_u^{(q)})^*\xi.
\]
\begin{lemma}[Multivariable approximate $q$-Toeplitz relation]\label{Lem=LqToeplitzDifferent}
Let $-1<q<1$ and $a,b\in[r]$. Then
\begin{equation}\label{Eqn=multiappQToeplitz}
  (L^{(q)}_{k,a})^*L^{(q)}_{k,b}
  = \delta_{ab}1
  +qL^{(q)}_{k,b}(L^{(q)}_{k,a})^*
  +E^{(q)}_{k,ab}.
\end{equation}
Here,
\[
\begin{aligned}
E^{(q)}_{k,ab}
={}
-\frac{q\delta_{ab}}{k}
 \sum_{u\in\Gamma_a}
 \ell_u^{(q)}(\ell_u^{(q)})^*+
( D^{(q)}_{k,b})^*
\left(
 \frac{B^{(q)}_{k,ba}}{k}\otimes1
\right)
 D^{(q)}_{k,a}.
\end{aligned}
\]
Consequently,
\begin{equation}\label{multiqToeperror}
 \max_{a,b\in[r]}\|E^{(q)}_{k,ab}\|
  \leq
  \frac{|q|\,\omega(\Gamma_k)}{k}
  +
  \frac{\omega(\Gamma_k)}{k}
  \|B_{\Gamma_k}^{(q)}\|.   
\end{equation}

In particular, there are events $\mathcal G_k$ with
$\mathbb P_\Gamma(\mathcal G_k)\longrightarrow1$
and deterministic numbers $\delta_k\rightarrow0$ such that
\[
\max_{a,b}\|E^{(q)}_{k,ab}\|\leq\delta_k
  \qquad(\Gamma\in\mathcal{G}_k).
\]
One may take
\[
  \delta_k=\mathcal O_{q,r}\left(\frac{\log k}{\sqrt{k}}\right).
\]
\end{lemma}

\begin{proof}
The graph-product and Clifford-twisted graph-product creation operators
satisfy, respectively,
\[
  \ell_u^*\ell_v
  =
  \delta_{uv}1
  +(A_{\Gamma_k})_{uv}\ell_v\ell_u^*
\]
and
\[
  (\widetilde{\ell}_u)^*\ell_v
  =
  \delta_{uv}1
  -(A_{\Gamma_k})_{uv}
\widetilde{\ell}_v(\widetilde{\ell}_u)^*.
\]
Thus, in both cases,
\[
  (\ell_u^{(q)})^*\ell_v^{(q)}
  =
  \delta_{uv}1
  +(A_{\Gamma_k}^{(q)})_{uv}
  \ell_v^{(q)}(\ell_u^{(q)})^*.
\]
Summing over $u\in\Gamma_{k,a}$ and $v\in\Gamma_{k,b}$ gives
\[
\begin{aligned}
(L^{(q)}_{k,a})^*L^{(q)}_{k,b}
={}&
\delta_{ab}1
+\frac1k
\sum_{\substack{u\in\Gamma_{k,a}\\v\in\Gamma_{k,b}}}
(A_\Gamma^{(q)})_{uv}
\ell_v^{(q)}(\ell_u^{(q)})^*.
\end{aligned}
\]
Subtracting
\[
  qL^{(q)}_{k,b}(L^{(q)}_{k,a})^*
  =
  \frac qk
  \sum_{\substack{u\in\Gamma_{k,a}\\v\in\Gamma_{k,b}}}
  \ell_v^{(q)}(\ell_u^{(q)})^*
\]
yields \eqref{Eqn=multiappQToeplitz}. The diagonal contribution occurs only when
$a=b$ and equals
\[
  -\frac qk
  \sum_{u\in\Gamma_{k,a}}
  \ell_u^{(q)}(\ell_u^{(q)})^*.
\]

Moreover,
\[
  \| D^{(q)}_{k,a}\|^2
  =\left\|
   \sum_{u\in\Gamma_{k,a}}
   \ell_u^{(q)}(\ell_u^{(q)})^*
  \right\|
  \leq\omega(\Gamma_k),
\]
and
\[
  \|B^{(q)}_{k,ba}\|
  \leq\|B_{\Gamma_k}^{(q)}\|.
\]
This proves \eqref{multiqToeperror}. The final assertion follows from the clique number
estimate and the random-matrix norm estimate for
$B_{\Gamma_k}^{(q)}$ from Theorem \ref{Lem=CliqueTailBound} and Theorem \ref{Thm=RandomMatrixTailBound} .
\end{proof}

\section{Vacuum Expectations Comparison  Bound}\label{sec:weakconvergence}
In this section we prove Theorem \ref{Thm=ConvergenceMomentsToeplitz} which shows in particular that $q$-Toeplitz generators $\ell_{q,a}$ can be approximated in moments by the models operators $L_{k,a}$. This strengthens a theorem by M\l otkowski  \cite{Mlot}  who showed already that $\ell_{q,a} + \ell_{q,a}^\ast$  can be approximated in moments by $L_{k,a} + L_{k,a}^\ast$. Our estimate is moreover quantitative.

We use the notation of  Section \ref{sec:preliminaries} and let $\Gamma$ be a graph with edge labels $[k] \times [r]$  and $\Gamma_i$ the subgraph with labels $[k] \times \{ a\}$.   
Let 
\[
L_{k,a} = \frac{1}{\sqrt{k}} \sum_{i=1}^k  \ell_{k,a}, \qquad  a=1,\cdots,r,
\]
be the averaged creation operators on the graph-product Fock space $\mathcal{F}_{\Gamma_k}$. 

Let
\[
  \Omega_{\Gamma_k}^{q}
  =
  \begin{cases}
    \Omega_{k},&q\geq0,\\
    \xi_{\mathrm{Cl},k}\otimes\Omega_{\Gamma_k},&q<0,
  \end{cases}
\]
where $\xi_{\mathrm{Cl},k}$ is the canonical tracial GNS vector
of the Clifford algebra. Define the vacuum states on two models:
\[
  \varphi_{q,k}(T)
  =
  \langle
    T\Omega_{\Gamma_k}^{q},
    \Omega_{\Gamma_k}^{q}
  \rangle,
  \qquad
  \varphi_q(T)
  =
  \langle T\Omega_q,\Omega_q\rangle.
\]

In Lemma \ref{Lem=LqToeplitzDifferent}, we proved that
\begin{equation}\label{appqToep}
    L_{k,a}^*L_{k,b}=\delta_{ab}I+qL_{k,b}L_{k,a}^*+E_{k,ab}.
    \end{equation}
Let $\delta_k=\max_{1\leq a,b\leq r}|| E_{k,ab}||$. For $\Gamma_k\sim \Gamma(kr,q)$, we have 
\[\delta_k=\mathcal{O}_{q,r}(\frac{\log k}{\sqrt{k}})\,
\]
in probability. 
On the other hand, we have 
\[\ell_{q,a}^*\ell_{q,b}=\delta_{ab}I+q\ell_{q,b}\ell_{q,a}^*.\]
Since $L_{k,a}^*\Omega_{\Gamma_k}^q=0,\ell_{q,a}^*\Omega_q=0$, we easily get
\[\varphi_{q,k}(L_{k,a_1}\cdots L_{k,a_s})=0,\qquad \varphi_q(\ell_{q,a_1}\cdots\ell_{q,a_s})=0,\]
for any nonempty pure creator word.  The analogous statement holds for a word consisting purely of   annihilators.

Since by  combining Theorem \ref{Thm=MainEstimate} and 
Lemma \ref{Lem=LqToeplitzDifferent},  and \cite[Lemma 4]{BozejkoSpeicher}, 
\[
||L_{k,a}||^2\leq \frac{1}{\sqrt{\min\{1-q,1\}}} +\delta_k,\qquad ||\ell_{q,a}||\leq \frac{1}{\sqrt{\min\{1-q,1\}}},\]
there is  a uniform bound $K_q$ such that for all $k,a$ we have  $||L_{k,a}||\leq K_q$ with probability going to 1 as $k \rightarrow \infty$, and $||\ell_{q,a}||\leq K_q$. 

Let $w$ be a word of letters $\{1,\cdots, r,1^*,\cdots, r^*\}$, and letter $a$ means a creator, $a^*$ means an annihilator. Let  $L_{k,w}$ be the corresponding word of the graph product averaged creators/annihilators, $\ell_{q,w}$ be the corresponding word of the $q$-Fock space creators/annihilators. For instance, if $w=a_1a_2a_3^\ast a_4^\ast$, then $L_{k,w}=L_{k,a_1}L_{k,a_2}L_{k,a_3}^*L_{k,a_4}^*$, $\ell_{q,w}=\ell_{q,a_1}\ell_{q,a_2}\ell_{q,a_3}^*\ell_{q,a_4}^*$. When $w=\varnothing$, we just put $L_{k,\varnothing}=I, \ell_{q,\varnothing}=I$. For these two kinds of words, we prove the following word comparison bound, which tell that the difference between $\varphi_{q,k}(L_{k,w})$ and $\varphi_q(\ell_{q,w})$ will  go to 0 being of   order  $\delta_k$.

\begin{theorem}\label{Thm=ConvergenceMomentsToeplitz}There are events $\mathcal G_k$ with
$\mathbb P_\Gamma(\mathcal G_k)\longrightarrow1$ as $k \rightarrow \infty$, such that  for any word $w$ of $\{1,\cdots,r,1^*,\cdots,r^*\}$, we have
    \[|\varphi_{q,k}(L_{k,w})-\varphi_q(\ell_{q,w})|\leq |w|C_q^{|w|}\delta_k.\]    
    Here $C_q$ is a constant depending only on $q$.
\end{theorem}
\begin{proof}
    For $t\in\mathbb{N}$, we define 
    \[D(t):=\sup_{|w|\leq t}|\varphi_{q,k}(L_{k,w})-\varphi_q(\ell_{q,w})|.\]
    Denote $Q_q:=\frac{1}{1-|q|}$. Let $K_q$ be as in the paragraph preceding this theorem and assume we are in the event that $\Vert L_{k,a} \Vert \leq K_q$ which has probability converging to 1 (Lemma \ref{Lem=LqToeplitzDifferent}). 
    Choose a large $C_q$ such that 
    \begin{equation}\label{Eqn=KqQqAssumption}
    K_q\leq C_q,\qquad Q_q\leq {C_q^2}.
    \end{equation}
    In what follows we prove by induction   that $D(t)\leq tC_q^t\delta_k$.

    For $t=0,1$, $D(t)=0$.

    For $|w|=t\geq 2$, we assume $D(t-2)\leq (t-2)C_q^{t-2}\delta_k$ and run 
    the following procedure to reduce the word length:

    If $w$ ends in an annihilator, then since $L_{k,a}^*\Omega_{\Gamma_k}^q=0,\ell_{q,a}^*\Omega_q=0$,  both traces vanish. If $w$ has no annihilator and is nonempty, both traces again vanish. 

    Therefore we can assume $w$ has at least one annihilator and does not end in annihilator. Write $w$ of the form $w=\tilde{w}a^*a_1a_2\cdots a_s$, where $a^*$ is the rightmost annihilator of $w$, and $\tilde{w}$ is the subword of $w$ before $a^*$.

    Then $L_{k,w}=L_{k,\tilde{w}}L_{k,a}^*L_{k,a_1}\cdots L_{k,a_s}$. Now, we move $L_{k,a}^*$ to the right using the approximating $q$-Toeplitz relation \cref{appqToep} iterately, then we obtain
    \begin{align*}
        L_{k,a}^*L_{k,a_1}\cdots L_{k,a_s}=&\sum_{i=1}^s q^{i-1}\delta_{aa_i}L_{k,a_1}\cdots L_{k,a_{i-1}}L_{k,a_{i+1}}\cdots L_{k,a_s}
    \\&+\sum_{i=1}^{s}q^{i-1}L_{k,a_1}\cdots L_{k,a_{i-1}}E_{k,aa_i}L_{k,a_{i+1}}\cdots L_{k,a_s}+q^s L_{k,a_1}\cdots L_{k,a_s}L_{k,a}^*.
    \end{align*}
    But $\varphi_{q,k}(L_{k,\tilde{w}}L_{k,a_1}\cdots L_{k,a_s}L_{k,a}^*)=0$ since $L_{k,a}^*\Omega_{\Gamma_k}^q=0$, therefore
    \begin{align*}
        \varphi_{q,k}(L_{k,w})=&\overbrace{\sum_{i=1}^s q^{i-1}\delta_{aa_i}\varphi_{q,k}(L_{k,\tilde{w}}L_{k,a_1}\cdots L_{k,a_{i-1}}L_{k,a_{i+1}}\cdots L_{k,a_s})}^{\textrm{Contraction Terms}}
        \\&+\overbrace{\sum_{i=1}^sq^{i-1}\varphi_{q,k}(L_{k,\tilde{w}}L_{k,a_1}\cdots L_{k,a_{i-1}}E_{k,aa_i}L_{k,a_{i+1}}\cdots L_{k,a_s})}^{\textrm{Error Terms}}.
    \end{align*}
    On the other hand, by the analogous arugment for the $q$-creation operators 
    \[\varphi_q(\ell_{q,w})=\sum_{i=1}^s q^{i-1}\delta_{aa_i}\tau_q(\ell_{\tilde{w}}\ell_{q,a_1}\cdots \ell_{q,a_{i-1}}\ell_{q,a_{i+1}}\cdots \ell_{q,a_s}).\]

    First, we bound the difference of the contraction terms of two models. For each $i$, the contraction removes two letters  $a_i$  and $a$ from $w$, thus the resulting words have length $t-2$, therefore the difference is bounded by $D(t-2)$. Then the difference between two total contraction terms is bounded by 
    \[\sum_{i=1}^s |q|^{i-1}D(t-2)\leq \frac{1}{1-|q|}D(t-2).\]
    Second, we bound the error terms of the graph product model. In $i$-th error term,  $L_{k,a}^*L_{k,a_i}$  is replaced by $E_{k,aa_i}$, so there are at most $t-2$ ordinary $L$ and $L^*$ factors. Thus
    \begin{align*}
        |\varphi_{q,k}(L_{k,\tilde{w}}L_{k,a_1}\cdots L_{k,a_{i-1}}E_{k,aa_i}L_{k,a_{i+1}}\cdots L_{k,a_i})|&\leq ||L_{k,\tilde{w}}L_{k,a_1}\cdots L_{a_{i-1}}E_{k,aa_i}L_{k,a_{i+1}}\cdots L_{k,a_s}||
        \\&\leq K_q^{t-2}\delta_k.
    \end{align*}
    Therefore the total error is bounded by \[\sum_{i=1}^s|q|^{i-1}K_q^{t-2}\delta_k\leq \frac{K_q^{t-2}}{1-|q|}\delta_k.\]

    Combining these two bounds, we get that
    \begin{equation}\label{recursionineq}
        D(t)\leq \frac{1}{1-|q|}D(t-2)+\frac{K_q^{t-2}}{1-|q|}\delta_k.
    \end{equation}
    Now, we use the induction assumption and the inequality \cref{recursionineq}, we have, by \eqref{Eqn=KqQqAssumption},  
    \[D(t)\leq Q_q(t-2) C_q^{t-2}\delta_k+Q_qK_q^{t-2}\delta_k\leq (t-2)C_q^t\delta_k+C_q^t\delta_k\leq tC_q^t\delta_k,\]
which concludes the proof. 
    \end{proof}

\begin{remark}[A non-backtracking viewpoint]
\label{rem:nonbacktracking}
The reduction in the proof of
Theorem~\ref{Thm=ConvergenceMomentsToeplitz} admits a useful
non-backtracking interpretation. On the graph-product Fock space
$\mathcal F_\Gamma=\ell^2(A_\Gamma^+)$, a creation operator follows a
labelled edge of the Cayley graph of the right-angled Artin monoid,
whereas the corresponding annihilation operator traverses such an
edge backwards. The relation
\[
 (\ell_u^{(q)})^*\ell_v^{(q)}
 =
 \delta_{uv}1+(A_\Gamma^{(q)})_{uv}
 \ell_v^{(q)}(\ell_u^{(q)})^*
\]
either contracts an immediate creation--annihilation reversal or moves
the backward step past a creation operator whenever there is an edge in the graph. The Clifford twist supplies the appropriate sign
when $q<0$.

After averaging, this relation becomes
\[
 L_{k,a}^*L_{k,b}= \delta_{ab}1+qL_{k,b}L_{k,a}^*+E_{k,ab}.
\]
Thus, moving the rightmost annihilator through the creators to its
right either contracts a matched pair, reducing the word length by
two, or moves the annihilator one step further with weight $q$. The
terminal uncontracted term has zero vacuum expectation. Iterating the
procedure produces the usual $q$-weighted Wick contractions
\cite{EffrosPopa}, while $E_{k,ab}$ collects the centered-adjacency and
finite-graph corrections. In this sense, the estimate
$\max_{a,b}\|E_{k,ab}\|\to0$ is a Fock-space analogue of controlling
the remainder in a non-backtracking expansion.

This viewpoint is related to the reduced-word framework underlying
graph-product Khintchine inequalities \cite{CKL20} and to the
matrix-valued and operator-valued non-backtracking methods used in
strong-convergence problems
\cite{BordenaveCollinsNB,BordenaveCollinsRep,FriedmanPuder}.
We emphasize, however, that this is an operator-theoretic analogy:
our proof does not introduce a Hashimoto operator or use an
Ihara--Bass identity.
\end{remark}
\section{Strong convergence for multivariate polynomials}\label{sec:multistrongconvergence}
In this section we prove that the reduced $q$-Toeplitz algebra admits strongly convergent models.
We let $-1 < q < 1$ and consider the $q$-Toeplitz algebra$\mathcal{T}_{q,r}$ with $r$ canonical generators $T_{1}, \ldots, T_{r}$ which subject to the $q$-Toeplitz relation
\[
T_a^\ast T_b = q T_b T_a^\ast + \delta_{ab}. 
\]
 Let again $\ell_{q,1}, \ldots, \ell_{q,r}  \in \mathcal{B}(\mathcal{F}_q(\mathbb C^r))$ denote the $q$-creation operators and let $\mathcal{T}_{q,r}^{\rm red}$ be the reduced $q$-Toeplitz C$^\ast$-algebra  generated by them. The assignment 
\begin{equation}\label{Eqn=Canonical}
    T_a \mapsto \ell_{q,a}
\end{equation}
is a faithful representation $\mathcal{T}_{q,r} \rightarrow \mathcal{T}_{q,r}^{\rm red}$. For the universal $q$-Toeplitz C$^\ast$-algebra $\mathcal{T}_{q,r}^{u}$ generated by $\mathcal{T}_{q,r}$, let 
\[
\Phi_q: \mathcal{T}_{q,r}^{u}  \rightarrow \mathcal{T}_{q,r}^{\rm red},
\]
be the map extending  \eqref{Eqn=Canonical} by  the universal property of $\mathcal{T}_{q,r}^{u}$.

\begin{remark}
    In \cite{PuszWoronowicz} it was proved that  $\Phi_q$ is an isomorphism for $\vert q \vert < \sqrt{2} -1$.  Outside this range this is unknown and we refer to  \cite{DykemaNica, Kuzmin} for related results and further discussion on this. 
\end{remark}

 Let again  $L_{k,a} =  \frac{1}{\sqrt{k}} \sum_{v=1}^k  \ell_{k,a} \in \mathcal{B}(\mathcal{F}_\Gamma)$, see Section \ref{sec:weakconvergence}.    For $d\geq0$, put
\[
  \mathcal E_d^{(q)}
  =
  \mathrm{span}
  \left\{
    1,\,
    s^{(q)}_{i_1}\cdots s^{(q)}_{i_m}
    :
    1\leq m\leq d,\;
    i_1,\ldots,i_m\in[r]
  \right\}
  \subseteq\mathcal T_{q,r},
\]
where $s^{(q)}_a=\ell_{q,a}+\ell_{q,a}^*$.
For a graph $\Gamma$, set
$S^{(q,\Gamma)}_{k,a}
  =
  L^{(q,\Gamma)}_{k,a}
  +(L^{(q,\Gamma)}_{k,a})^*$.

\begin{remark}\label{Rmk=LinIn}
The monomials occurring in the definition of
$\mathcal E_d^{(q)}$ are linearly independent.  This is proved in (the proof of) \cite[Theorem 4.1]{CIW} but we give the short argument for convenience of the reader.   Indeed, if $P$
has degree $m$, then the projection of
\( P(s^{(q)})\Omega_q\)
onto the $m$-particle space is the coefficient tensor of the
homogeneous degree-$m$ part of $P$. The $q$-inner product is
strictly positive for $|q|<1$, so this tensor vanishes only when
all degree-$m$ coefficients vanish. Descending induction to the degree of $P$ proves
the assertion.
\end{remark}

As a consequence of Remark \ref{Rmk=LinIn}, there is a well-defined linear map
$\Theta_{k,\Gamma}^{(d)}
  :
  \mathcal E_d^{(q)}
  \longrightarrow \mathcal{B}(\mathcal F_\Gamma)$
given by
\[
  \Theta_{k,\Gamma}^{(d)}
  \left(
    s^{(q)}_{i_1}\cdots s^{(q)}_{i_m}
  \right)
  =
  S^{(q,\Gamma)}_{k,i_1}
  \cdots
  S^{(q,\Gamma)}_{k,i_m}.
\]
We include the following standard lemma for completeness. 

\begin{lemma}\label{Lem=BtoCB} 
Let $\mathcal{E}$ be a finite dimensional operator space. There exists  $c_{\mathcal{E}} > 0$ such that for every $R: \mathcal{E} \rightarrow \mathcal{B}(\mathcal{H})$ we have  $\Vert R \Vert_{{\rm cb}} \leq c_{\mathcal{E}} \Vert R \Vert$.  
\end{lemma}
\begin{proof}
Let $e_1, \ldots, e_m$ be a basis of $\mathcal{E}$ and let $f_1^\ast, \ldots, f_m^\ast \in \mathcal{E}^\ast$ be the dual basis, i.e. $f_a^\ast(e_b) = \delta_{ab}$. We have $\Vert f_a^\ast \Vert = \Vert f_a^\ast \Vert_{{\rm cb}}$. Now let $X \in M_D(\mathbb{C})  \otimes \mathcal{E}$ and write $X = \sum_{a =1}^m  X_a \otimes e_a$. As $X_a = ({\rm id} \otimes f_a^\ast)(X)$ we find that 
\[
\Vert ({\rm id} \otimes R)(X) \Vert = \Vert  \sum_{a =1}^m  X_a \otimes R(e_a)  \Vert
\leq \sum_{a=1}^m \Vert X_a \Vert \Vert R \Vert \Vert e_a \Vert \leq \sum_{a=1}^m \Vert f_a^\ast \Vert \Vert X \Vert \Vert R \Vert \Vert e_a \Vert, 
\]
so that $c_{\mathcal{E}} = \sum_{a=1}^m \Vert f_a^\ast \Vert  \Vert e_a \Vert$ does the job. 
\end{proof}

For the upper bound in the following theorem we we were inspired by \cite[Theorem 6]{ShlyakhtenkoStrongC}. We show that in the nuclear case the argument can be changed to yield a completely bounded strong version of the strong convergence result. 

\begin{theorem}[Complete bounded-degree strong convergence]\label{Thm=StrongConvergence}
Let $-1<q<1$ satisfy
\(|q|<\sqrt2-1\),
and fix $r,d\geq1$. There are  events $\mathcal{G}_k$  such that $ \mathbb P_\Gamma(\mathcal{G}_k)\rightarrow1$ 
and
\[
\begin{aligned}
\sup_{\Gamma\in\mathcal{G}_k}
\sup_{D\geq1}
\sup_{0\neq X\in M_D(\mathcal E_d^{(q)})}
\left|
 \frac{
  \|(\id_{M_D}\otimes\Theta_{k,\Gamma}^{(d)})(X)\|
 }{
  \|X\|
 }
 -1
\right|
\rightarrow 0.
\end{aligned}
\]
Equivalently, for every sequence $D_k\geq1$ and every sequence
of nonzero matrix-valued noncommutative polynomials $P_k\in
  M_{D_k}(\mathbb C)
  \otimes
  \mathbb C\langle x_1,\ldots,x_r\rangle$
of degree at most $d$,
\[
  \frac{
   \|P_k(S^{(q,\Gamma_k)}_{k,1},
          \ldots,
          S^{(q,\Gamma_k)}_{k,r})\|
  }{
   \|P_k(s^{(q)}_1,\ldots,s^{(q)}_r)\|
  }
  \rightarrow 1
\]
in probability.
\end{theorem}

\begin{proof}
Let $\mathcal{G}_k$ be the events from Lemma \ref{Lem=LqToeplitzDifferent}, enlarged if
necessary so that
\[
  \max_{1\leq a\leq r}
  \|L^{(q,\Gamma)}_{k,a}\|
  \leq K_q
  \qquad(\Gamma\in\mathcal{G}_k)
\]
for a constant $K_q$ independent of $k$ and $\Gamma \in \mathcal{G}_k$, and  we have
$\mathbb P_\Gamma(\mathcal{G}_k)\longrightarrow1$.

We first establish fixed-matrix-level convergence. Let
$\Gamma_k\in\mathcal{G}_k$ be an arbitrary deterministic
sequence and let $\mathcal U$ be a nonprincipal ultrafilter.
Put
\[
  \mathcal A_k
  =
  C^*(L^{(q,\Gamma_k)}_{k,1}, 
      \ldots,
      L^{(q,\Gamma_k)}_{k,r}),
\] 
as C$^\ast$-subalgebras of $\ell^\infty(\mathcal{G}_k; B(\mathcal{H}_\Gamma))$. 
 By \eqref{appqToep}   the elements 
\[
  \mathcal{L}_a=(L^{(q,\Gamma_k)}_{k,a})_{k,\mathcal U}
  \in
  \prod_{k,\mathcal U}\mathcal A_k
\]
satisfy the exact $q$-Toeplitz relations
\[
  \mathcal{L}_a^*\mathcal{L}_b=
  \delta_{ab}1+q\mathcal{L}_b\mathcal{L}_a^*.
\]
Therefore the universal property gives a representation of the
universal $q$-Toeplitz algebra. Since
$ \Phi_q: \mathcal{T}_{q,r}^{u} \longrightarrow  \mathcal{T}_{q,r}^{\rm red}$
is an isomorphism for $|q|<\sqrt2-1$ \cite{PuszWoronowicz}, this representation
induces a $*$-homomorphism
\begin{equation}\label{repultra}
    \psi_{\mathcal U}
  :
  \mathcal{T}_{q,r}^{\rm red}
  \longrightarrow
  \prod_{k,\mathcal U}\mathcal A_k,
  \qquad
  \psi_{\mathcal U}(\ell_{q,a})=\mathcal{L}_a.
\end{equation}

Let, for $0 \leq q < 1$,  $\mathcal F_{\mathcal U} = \prod_{k, \mathcal{U}} \mathcal{F}_{\Gamma(k,q)}$ be the corresponding Hilbert-space
ultraproduct and let $\Omega_{\mathcal U}$ be the ultraproduct
vacuum vector. In the Clifford-twisted case, $-1 \leq q < 0$, use the canonical
tracial GNS vector in the Clifford factor to define   $\mathcal F_{\mathcal U}$. Theorem \ref{Thm=ConvergenceMomentsToeplitz} gives,
for every $*$-polynomial $Q$,
\[
\begin{aligned}
\left\langle
 Q(\mathcal{L}_1,\ldots,\mathcal{L}_r,\mathcal{L}_1^*,\ldots,\mathcal{L}_r^*)
 \Omega_{\mathcal U},
 \Omega_{\mathcal U}
\right\rangle
=
\left\langle
 Q(\ell_{q,1},\ldots,\ell_{q,r},
   \ell_{q,1}^*,\ldots,\ell_{q,r}^*)
 \Omega_q,
 \Omega_q
\right\rangle.
\end{aligned}
\]
Hence the cyclic representation generated by
$\Omega_{\mathcal U}$ is unitarily equivalent to the $q$-Fock
representation. The latter is faithful on the reduced algebra
 $\mathcal{T}_{q,r}^{\rm red}$; this easily follows as $\Omega_q$ is cyclic. Therefore $\psi_{\mathcal U}$ is injective,
and hence completely isometric.  

It follows that, for every fixed $m$,
\begin{equation}\label{supsup}
    \begin{aligned}
\sup_{\Gamma\in\mathcal{G}_k}
\sup_{\substack{
 X\in M_m(\mathcal E_d^{(q)})\\
 \|X\|\leq1
}}
\left|
 \|(\id_{M_m}\otimes\Theta_{k,\Gamma}^{(d)})(X)\|
 -\|X\|
\right|
\rightarrow 0.
\end{aligned}
\end{equation}

Indeed, otherwise one could choose a violating sequence
$\Gamma_k$ and, by compactness of the unit ball of
$M_m(\mathcal E_d^{(q)})$, obtain a contradiction to the
complete isometry of (5.4).

We next prove the upper estimate at arbitrary matrix level.
Suppose it fails. Then, after passing to a subsequence, there
are
\[
  \Gamma_k\in\mathcal{G}_k,\qquad
  D_k\geq1,\qquad
  X_k\in M_{D_k}(\mathcal E_d^{(q)})
\]
such that
\(
  \|X_k\|=1
\)
and
\begin{equation}\label{geq1epl}
     \|
   (\id_{M_{D_k}}\otimes
    \Theta_{k,\Gamma_k}^{(d)})(X_k)
  \|
  \geq1+\varepsilon
\end{equation}

for some $\varepsilon>0$.

The reduced $q$-Toeplitz algebra $\mathcal{T}_{q,r}^{\rm red}$ is nuclear:
by \cite[Theorem~1.2]{Kuzmin}, it is isomorphic to the
Cuntz--Toeplitz algebra.   Recall that $\prod_{k, \mathcal{U}} \mathcal A_k$ is by definition the quotient of $\prod_k \mathcal A_k$ by the $c_0$ sequences with respect to $\mathcal{U}$. By the Choi--Effros lifting theorem,
the homomorphism \eqref{repultra} has a completely positive contractive
lift  
\begin{equation}\label{repultra1}
    \widetilde\psi
  =
  (\phi_k)_k
  :
    \mathcal{T}_{q,r}^{\rm red}
  \longrightarrow
  \prod_k\mathcal A_k.
\end{equation}

Set
\[
  R_k
  =
  \Theta_{k,\Gamma_k}^{(d)}
  -
  \phi_k|_{\mathcal E_d^{(q)}}.
\]
Since \eqref{repultra1} lifts \eqref{repultra} and $\mathcal E_d^{(q)}$ is finite dimensional we have  $\Vert R_k \Vert \rightarrow 0$ and by Lemma \ref{Lem=BtoCB} even  $\Vert R_k \Vert_{{\rm cb}} \rightarrow 0$. 
Consequently,
\[
\begin{aligned}
\|
 (\id_{M_{D_k}}\otimes
  \Theta_{k,\Gamma_k}^{(d)})(X_k)
\|
&\leq
\|
 (\id_{M_{D_k}}\otimes\phi_k)(X_k)
\|
+\|R_k\|_{\rm cb}\|X_k\|\\
&\leq
1+ \order_{\mathcal U}(1),
\end{aligned}
\]
contradicting \eqref{geq1epl}.

For the lower estimate, suppose that there are
$\varepsilon>0$ and sequences as above such that
\[
  \|X_k\|=1,
  \qquad
  \|
   (\id_{M_{D_k}}\otimes
    \Theta_{k,\Gamma_k}^{(d)})(X_k)
  \|
  \leq1-\varepsilon.
\]
Fix $\delta\in(0,\varepsilon/3)$. Since
$\mathcal{T}_{q,r}^{\rm red}$ is exact  \cite{KennedyNica}, in fact it  is even nuclear \cite{Kuzmin}, the fixed-family compression lemma
\cite[Lemma~5.13]{MageedelaSalle}, applied to $E=\mathcal E_d^{(q)}$, gives an integer
$m_0=m_0(d,r,q,\delta)$ and contractions
\[
  U_k,V_k\in M_{m_0,D_k}(\mathbb C)
\]
such that
\[
  Y_k
  =
  (U_k\otimes1)X_k(V_k^*\otimes1)
  \in M_{m_0}(\mathcal E_d^{(q)})
\]
satisfies
\[
  \|Y_k\|\geq1-\delta.
\]
On the model side,
\begin{equation}\label{1minepl}
    \begin{aligned}
\|
 (\id_{M_{m_0}}\otimes
  \Theta_{k,\Gamma_k}^{(d)})(Y_k)
\|\leq
\| (\id_{M_{D_k}}\otimes
  \Theta_{k,\Gamma_k}^{(d)})(X_k)
\|\leq1-\varepsilon.
\end{aligned}
\end{equation}

But \eqref{supsup}, applied at the fixed  $m_0$, gives
\[
  \|
   (\id_{M_{m_0}}\otimes
    \Theta_{k,\Gamma_k}^{(d)})(Y_k)
  \|
  =
  \|Y_k\|+o_{\mathcal U}(1)
  \geq1-\delta+ o_{\mathcal U}(1),
\]
contradicting \eqref{1minepl}.
\end{proof}

\begin{corollary}
For every fixed matrix-valued polynomial $P$,
\[
\|P(S^{(q,\Gamma)}_{k,1},\ldots,S^{(q,\Gamma)}_{k,r})\|
\rightarrow
\|P(s^{(q)}_1,\ldots,s^{(q)}_r)\|
\]
in probability.
\end{corollary}

\begin{proof}
Apply Theorem~5.1 with
$d=\deg P$ and the fixed coefficient dimension of $P$.
\end{proof}

\begin{remark}
Theorem \ref{Thm=StrongConvergence} holds true for any $-1 <  q < 1$ as long as $\Phi_q$ is an isomorphism.   
\end{remark} 

\section{Strongly convergent matrix models}\label{sec:strongconvergencematrix}

In Section \ref{sec:multistrongconvergence} we have shown that $q$-Toeplitz algebras admit strongly convergent models of random graph products  of creation and annihilation operators with arbitrary matrix amplifications. The latter models are still infinite dimensional. In this section we focus on the subalgebra of $q$-semicircular elements and show that they can be approximated with GUE's yielding finite dimensional models that converge strongly. This is also the place where control of the growth rate of the matrix amplification is essential.

\subsection{Tensor-GUE models}

Tensor-GUE models are random operators on the multipartite space
$(\mathbb C^N)^{\otimes L}$ in which each independent GUE matrix acts
on a prescribed collection of tensor coordinates and as the identity
on the complementary coordinates. The overlap pattern of these
supports determines the limiting mixed independence: disjoint supports
give commuting, classically independent variables, while the full joint
limit is the  graph-product semicircular family. This form of mixed classical and free independence
is discussed in \cite{Mlot, SpeicherWyso}. Weak convergence of the corresponding
tensor-matrix models  was proved in
\cite[Theorem~4]{CharlesworthCollins}. Strong convergence was obtained
for particular tensor configurations in \cite{BelCap,CY26}, and for
every fixed finite interaction pattern in
\cite[Section~9.4, Theorem~9.8]{CGVvH}.

We recall the construction and comparison scheme of
\cite[Section~9.4]{CGVvH}. The fixed-parameter strong-convergence
theorem is not by itself sufficient for our application, because the
matrix-coefficient dimension below grows with $N$. \cref{subsec:uniform-strong-tensor-gue}
establishes the required uniform refinement.

The results of this section can be found in \cite{CGVvH}. 

Fix integers \(L,V\geq 1\) and nonempty subsets
$K_1,\ldots,K_V\subseteq [L]:=\{1,\ldots,L\}$.
For \(N\geq 1\), let
\[
G_v^{(N)}=
H_v^{(N)}\otimes I_{N^{\,L-|K_v|}}
\in M_{N^L}(\mathbb C),
\qquad 1\leq v\leq V,
\]
where \(\bigl(H_v^{(N)}\bigr)_{v=1}^V\) are independent normalized GUE matrices of dimension $N^{|K_v|}$. Set
$G^{(N)}=\bigl(G_1^{(N)},\ldots,G_V^{(N)}\bigr)$.
Its limiting model is the \(\Gamma\)-independent semicircular family
$x^{\Gamma}=(x_1^{\Gamma},\ldots,x_V^{\Gamma})$, where \(\Gamma\) is the graph on \([V]\) defined by
\[
v\sim_\Gamma w
\quad\Longleftrightarrow\quad
K_v\cap K_w=\varnothing.
\] 
For self-adjoint coefficient matrices
 $A_0,\ldots,A_V\in M_D(\mathbb C)_{\mathrm{sa}}$,
define the self-adjoint linear pencil
\[
\mathcal{P}_A(x) =A_0\otimes 1+\sum_{v=1}^V A_v\otimes x_v.
\]
And we denote  For \(T\geq 1\), define
\[
G_v^{(N,T)}=
\frac{1}{\sqrt{T}}\sum_{t=1}^T
\left(\bigotimes_{j\in K_v} G_{v,t,j}^{(N)}\right)
\otimes I_{N^{L-|K_v|}},
\]
where all matrices \(G_{v,t,j}^{(N)}\) are independent normalized \(N\times N\) GUE matrices. Write
\[
G^{(N,T)}=
\bigl(G_1^{(N,T)},\ldots,G_V^{(N,T)}\bigr).
\]
Correspondingly, let
\[
s_v^{(T)}=
\frac{1}{\sqrt{T}}
\sum_{t=1}^T\left(
\bigotimes_{j\in K_v} s_{v,t,j}
\right)\otimes 1,
\]
where the \(s_{v,t,j}\)'s form the Hayes' limiting model, and set
$s^{(T)}=\bigl(s_1^{(T)},\ldots,s_V^{(T)}\bigr)$. 

By \cite[Lemma~9.9]{CGVvH}, $G^{(N,T)}$ and $G^{(N)}$
have the same entrywise mean and covariance (for the real and
imaginary parts of their entries). Moreover,
\cite[Lemma~9.10]{CGVvH} gives, for every fixed coefficient dimension
and every fixed linear pencil $P_A$,
\[
  \|P_A(s^{(T)})\|
  \longrightarrow
  \|P_A(x^\Gamma)\|,
  \qquad T\to\infty.
\]
Below, the quantitative Hayes-model estimate
\cite[Theorem~9.7]{CGVvH} and the universality principle
\cite[Theorem~2.8]{BrailovskayaVanHandel2024} are combined to make this comparison uniform
over growing matrix-coefficient dimensions.

\subsection{Uniform large coefficient dimensions strong convergence of tensor-GUE}\label{subsec:uniform-strong-tensor-gue}

\begin{definition}[Dimension-free limiting CLT modulus]
    For \(M\geq 1\), define 
\[
\Delta_{L,V,M}(T):=
\sup_{\substack{D\geq 1,\;
A_0,\ldots,A_V\in M_D(\mathbb C)_{\mathrm{sa}}\\
\Lambda(A)\leq M
}}
\left|\|\mathcal P_A(s^{(T)})\|-\|\mathcal P_A(x^{\Gamma})\|
\right |,
\]
\end{definition}

The following lemma tells why $\Delta_{L,V,M}(T)$ is called dimension-free: it is precisely what removes the matrix-coefficient dimension
from the limiting \(T\to\infty\) step.
\begin{lemma}\label{asytompzeroCLTmodulus}

For fixed $L,V,M$ and every $T\geq 1$, we have
\[
\Delta_{L,V,M}(T)\leq \frac{(2^L-2)M}{\sqrt T}.
\]
Consequently, $\Delta_{L,V,M}(T)\to 0$ as $T\to\infty$.
\end{lemma}

\begin{proof}
We use the full Fock-space realization of the free semicircular
families appearing in the Hayes' limiting model.  For each
$j\in[L]$, let
\[
\mathcal F_j
   =\mathcal F\bigl(\ell_2([V]\times[T])\bigr),
\]
and denote by $a^{(j)}_{v,t}$, $v\in[V]$, $t\in[T]$
the corresponding left creation operators.  Thus $s_{v,t,j}=a^{(j)}_{v,t}+a^{(j)*}_{v,t}$.
We work on $\mathcal H_T=\bigotimes_{j=1}^L\mathcal F_j$.
For $v\in[V]$, put
\[
W^{(T)}_{v,t}=
   \bigotimes_{j=1}^L b^{(v,t)}_j,
   \qquad
   b^{(v,t)}_j=
   \begin{cases}
   a^{(j)}_{v,t},&j\in K_v,\\
   1,&j\notin K_v,
   \end{cases}
\]
and define
\[
\lambda_v^{(T)} =\frac1{\sqrt T}\sum_{t=1}^T W^{(T)}_{v,t}.
\]
We first identify the joint operator model of
$(\lambda_v^{(T)})_{v=1}^V$.  Since $K_v$ is nonempty, $W^{(T)*}_{v,t}W^{(T)}_{v,u}=\delta_{tu}1$,
and hence $\lambda_v^{(T)*}\lambda_v^{(T)}=1$.
Moreover, if $u\neq v$ and $K_u\cap K_v\neq\varnothing$, then
a tensor coordinate in $K_u\cap K_v$ contains the factor $a^{(j)*}_{u,t}a^{(j)}_{v,r}=0$,
so that $\lambda_u^{(T)*}\lambda_v^{(T)}=0$.
If $K_u\cap K_v=\varnothing$, the two operators act on disjoint
tensor coordinates and therefore doubly commute:
\[
\lambda_u^{(T)}\lambda_v^{(T)}
   =\lambda_v^{(T)}\lambda_u^{(T)},
\qquad
\lambda_u^{(T)*}\lambda_v^{(T)}
   =\lambda_v^{(T)}\lambda_u^{(T)*}.
\]
Thus the family $(\lambda_v^{(T)})_{v=1}^V$ satisfies precisely the
graph-product creation relations associated with  
\[
u\sim_\Gamma v
   \quad\Longleftrightarrow\quad
K_u\cap K_v=\varnothing.
\]
For completeness, we verify that this representation is an
amplication of the graph Fock representation.  Let $P_v=\lambda_v^{(T)}\lambda_v^{(T)*}$.
The preceding relations imply that the projections $P_v$ commute:
if $u\not\sim_\Gamma v$, then $P_uP_v=0$, while if
$u\sim_\Gamma v$, this follows from double commutation.  Set
\[
Q_T=\prod_{v=1}^V(1-P_v),
\qquad
\mathcal K_T=Q_T\mathcal H_T.
\]
For $w\in \mathcal A_\Gamma^+$ denote by $\lambda_w^{(T)}$ the
corresponding product of the $\lambda_v^{(T)}$.  The graph-product
relations imply
\[
Q_T\lambda_w^{(T)*}\lambda_{w'}^{(T)}Q_T
   =\delta_{w,w'}Q_T.
\]
Indeed, using the normal form in the right-angled Artin monoid,
one first cancels the maximal common left divisor of $w$ and
$w'$. If residual letters remain, either a noncommuting pair
gives zero, or the graph relations move a residual creation or
annihilation operator next to $Q_T$.  The latter term vanishes
because
\[
Q_T\lambda_v^{(T)}=0,
\qquad
\lambda_v^{(T)*}Q_T=0.
\]
Therefore the map
\[
U_T:
\ell_2(\mathcal A_\Gamma^+)\otimes\mathcal K_T
   \longrightarrow \mathcal H_T,
\qquad
U_T(\delta_w\otimes\xi)=c_w^{(T)}\xi,
\]
is an isometry.
It is also onto.  Indeed, grade $\mathcal H_T$ by total Fock
degree.  Each $c_v^{(T)}$ raises the total degree by $|K_v|\geq1$,
whereas $P_v$ and $Q_T$ preserve the grading.  If
$\mathcal H_T^{(n)}$ denotes the homogeneous subspace of total
degree $n$, then
\[
(1-Q_T)\mathcal H_T^{(n)}
   \subseteq
   \sum_{v=1}^V P_v\mathcal H_T^{(n)}
   \subseteq
   \sum_{v=1}^V
   c_v^{(T)}\mathcal H_T^{(n-|K_v|)}.
\]
Induction on $n$ shows that every homogeneous subspace is
contained in the range of $U_T$.  Hence $U_T$ is unitary.

 Since we have $\lambda_v^{(T)}
   =
   U_T(\ell_v\otimes1_{\mathcal K_T})U_T^*$.
Consequently, for every coefficient dimension $D$,
\begin{equation}\label{6.1}
  \left\|
 A_0\otimes1+
 \sum_{v=1}^V
 A_v\otimes
 \bigl(\lambda_v^{(T)}+\lambda_v^{(T)*}\bigr)
\right\|=\left\|
 A_0\otimes1+
 \sum_{v=1}^V
 A_v\otimes(\ell_v+\ell_v^*)
\right\|=\|\mathcal P_A(x^{\Gamma})\|.  
\end{equation}

We now compare $s_v^{(T)}$ with
$\lambda_v^{(T)}+\lambda_v^{(T)*}$.  Expanding  
\[
\bigotimes_{j\in K_v}
\bigl(a^{(j)}_{v,t}+a^{(j)*}_{v,t}\bigr),
\]
in the expression 
\[
s_v^{(T)} := T^{-\frac{1}{2}} \sum_{t=1}^T  \bigotimes_{j\in K_v} \bigl(a^{(j)}_{v,t}+a^{(j)*}_{v,t}\bigr),
\]
the all-creation and all-annihilation terms give
$\lambda_v^{(T)}$ and $\lambda_v^{(T)*}$, respectively.  Thus
\[
s_v^{(T)} =
   \lambda_v^{(T)}+\lambda_v^{(T)*}+r_v^{(T)},
\]
where
\[
r_v^{(T)}=
  \sum_{\varnothing\neq S\subsetneq K_v}r_{v,S}^{(T)}
\]
and  
\[
r_{v,S}^{(T)}
 =\frac1{\sqrt T}\sum_{t=1}^T
 \bigotimes_{j=1}^L d^{(v,t,S)}_j,
\]
with
\[
d^{(v,t,S)}_j=
\begin{cases}
a^{(j)}_{v,t},&j\in S,\\
a^{(j)*}_{v,t},&j\in K_v\setminus S,\\
1,&j\notin K_v.
\end{cases}
\]
Fix a nonempty proper subset $S\subsetneq K_v$.  Because
$S\neq\varnothing$, the cross terms with distinct $t$ vanish in
$r_{v,S}^{(T)*}r_{v,S}^{(T)}$.  Hence
\[
r_{v,S}^{(T)*}r_{v,S}^{(T)}
 =\frac1T\sum_{t=1}^T
 \left(\bigotimes_{j\in S}1
 \right)\otimes
 \left(\bigotimes_{j\in K_v\setminus S} a^{(j)}_{v,t}a^{(j)*}_{v,t}
 \right)\otimes1.
\]
As $K_v\setminus S\neq\varnothing$, the summands in the last
display are pairwise orthogonal projections.  It follows that $\|r_{v,S}^{(T)}\|=\frac1{\sqrt T}$. There are $2^{|K_v|}-2$ nonempty proper subsets of $K_v$, so
\begin{equation}\label{6.2}
    \|r_v^{(T)}\|\leq
   \frac{2^{|K_v|}-2}{\sqrt T}
   \leq\frac{2^L-2}{\sqrt T}.
\end{equation}

Combining \eqref{6.1} and \eqref{6.2}, we obtain
\[
\begin{aligned}
\left|\|\mathcal P_A(s^{(T)})\|-\|\mathcal P_A(x^{\Gamma})\|
\right|\leq
\left\|\sum_{v=1}^V A_v\otimes r_v^{(T)}\right\|
&\leq\sum_{v=1}^V
 \|A_v\|\,\|r_v^{(T)}\|\\
&\leq\frac{2^L-2}{\sqrt T}
\sum_{v=1}^V\|A_v\|\leq\frac{(2^L-2)M}{\sqrt T}.
\end{aligned}
\]
The estimate is independent of $D$, and taking the supremum
proves the statement of the lemma.
\end{proof}
\begin{remark} 
    We record a stronger consequence of the preceding construction in Lemma~\ref{asytompzeroCLTmodulus}.     There is a faithful representation
\[
\pi_T:C^*(x^{\Gamma}_1,\ldots,x^{\Gamma}_V)\longrightarrow B(\mathcal H_T)
\]
such that
\[
\max_{1\leq v\leq V}
\|s_v^{(T)}-\pi_T(x^{\Gamma}_v)\|
\leq
\frac{2^L-2}{\sqrt T}. 
\]
Consequently, for every self-adjoint linear pencil with
$\Lambda(A)\leq M$,
\[
\left\|
 \mathcal P_A(s^{(T)})
 -(\id_{M_D}\otimes\pi_T)(\mathcal P_A(x^{\Gamma}))
\right\|
\leq
\frac{(2^L-2)M}{\sqrt T}.
\]
As $\id_{M_D}\otimes\pi_T$ is faithful, spectral stability under
self-adjoint perturbations  gives
\[
d_{\rm H}\!\left(
 \mathrm{sp}(\mathcal P_A(s^{(T)})),
 \mathrm{sp}(\mathcal P_A(x^{\Gamma}))
\right)
\leq
\frac{(2^L-2)M}{\sqrt T}. 
\]
\end{remark}

Let \(C_0\) be the universal constant in \cite[Theorem 9.7]{CGVvH}, and define
\begin{equation}\label{eq:cH}
c_H(L,V,T):=L^{-2}\bigl(C_0LVT\bigr)^{-2L}.
\end{equation}
The exponent \(L\) occurs because a linear pencil in the tensor variables
becomes a polynomial of degree at most \(L\) in the Hayes' variables, while the number of Hayes' variables is \(VT\). More precisely,
\cite[Theorem 9.7]{CGVvH} states that the quantitative constant satisfies
$c^{-1}=q_0^2(CLr)^{2q_0}$,
where \(q_0\leq L\) and \(r=VT\).

\begin{theorem}[Quantitative tensor-GUE transfer]\label{thm:quantitative-tensor-gue-transfer}
Fix \(L,V,M\). There exist constants
\[
C_1=C_1(L,V,M),\qquad
C_2=C_2(L,V),\qquad
c_2=c_2(L,V)>0,
\]
such that the following holds:

Let \(D,N,T\geq 1\), let \(t>0\), and let \(\rho\in(0,1]\). Suppose
\[
T\leq e^N
\]
and further,  with $c_H$ is defined in \eqref{eq:cH}, assume  
\begin{equation}\label{eq:D-bound-cH}
D\leq \exp\bigl(c_H(L,V,T)N\rho^2\bigr).
\end{equation}
Then for every self-adjoint linear pencil \(\mathcal P_A\) satisfying
\(\Lambda(A)\leq M\), we have
\begin{align}
&\mathbb P\Big[
 \|\mathcal P_A(G^{(N)})\| >
 \|\mathcal P_A(x^{\Gamma})\|+\Delta_{L,V,M}(T)+C_2M\rho
 \notag
 +C_1\big(
 N^{-1/2}t^{1/2}
 +T^{-1/12}t^{2/3}
 +T^{-1/4}t\big)
 \Big] \notag\\&\qquad\le
 2DN^Le^{-t}
 +C_1VT e^{-c_2N}
 +\frac{N^L}{c_H(L,V,T)\rho}
   e^{-c_H(L,V,T)N\rho^2}.
\label{eq:quantitative-tensor-gue-bound}
\end{align}
\end{theorem}

\begin{proof}
Write
\[
X_{A,N,T}
=
\mathcal{P}_A(G^{(N,T)})
=
A_0\otimes1+
\frac1{\sqrt T}
\sum_{v=1}^V\sum_{t=1}^T
A_v\otimes Z_{v,t},
\]
where each $Z_{v,t}$ is a tensor product of at most $L$
independent normalized $N\times N$ GUE matrices, amplified by
identities on the unused tensor coordinates. Its associated
Gaussian matrix with the same mean and covariance has the same
law as $\mathcal P_A(G^{(N)})$.

We estimate the parameters in the universality theorem
uniformly over all coefficient dimensions $D$ and all
coefficient tuples satisfying $\Lambda(A)\leq M$. Repeating the
estimates in the proof of \cite[Theorem~9.8]{CGVvH}, every
occurrence of the coefficient matrices is bounded by $\sum_{v=1}^V\|A_v\|\leq M$.
Consequently, there exist a constant $C_{L,V} > 0 $ such that (in the notation from \cite{BrailovskayaVanHandel2024})
\[
\sigma_*(X_{A,N,T})
\leq C_{L,V}MN^{-1/2},
\qquad
\sigma(X_{A,N,T})
\leq C_{L,V}M,
\]
and  
\[
\overline R(X_{A,N,T})
\leq C_{L,V}MT^{-1/2}.
\]
Moreover, there exist constants $c_{L,V}  > 0
$ such that  for $T\leq e^N$, 
\[
\mathbb P\left[
 \max_{v,t}\|Z_{v,t}\|>C_{L,V}
\right]
\leq
C_{L,V}VT e^{-c_{L,V}N}.
\]
Set $R_0=C_{L,V}M T^{-1/4}$,
where the constant is enlarged if necessary. Since
\[
\overline R(X_{A,N,T})
   \leq C_{L,V}MT^{-1/2}\quad\textrm{and}\quad \sigma(X_{A,N,T})\leq C_{L,V}M,
\]
we have
\[
R_0
\geq
\overline R(X_{A,N,T})^{1/2}
\sigma(X_{A,N,T})^{1/2}
+
\sqrt{2}\,\overline R(X_{A,N,T}).
\]
Moreover, on the event $\max_{v,t}\|Z_{v,t}\|\leq C_{L,V}$,
every independent summand satisfies
\[
\left\|
 \frac1{\sqrt T}A_v\otimes Z_{v,t}
\right\|
\leq
C_{L,V}M T^{-1/2}
\leq R_0.
\]
Consequently, the error term in the unbounded spectrum
universality theorem satisfies for $t >0$
\[
\begin{aligned}
&\sigma_*(X_{A,N,T})t^{1/2}
+
R_0^{1/3}\sigma(X_{A,N,T})^{2/3}t^{2/3}
+
R_0t\\
&\qquad\leq
C_1\left(
N^{-1/2}t^{1/2}
+
T^{-1/12}t^{2/3}
+
T^{-1/4}t
\right),
\end{aligned}
\]
where $C_1=C_1(L,V,M)$.
Therefore the Brailovskaya--van Handel universality comparison (\cite{BrailovskayaVanHandel2024}),
applied as in the proof of \cite[Theorem~9.8]{CGVvH}, gives
\[
\mathrm{sp}(\mathcal P_A(G^{(N)}))
\subseteq
\mathrm{sp}(\mathcal P_A(G^{(N,T)}))+[-u,u]
\]
except on an event of probability at most
\[
2DN^Le^{-t}+C_1VT e^{-c_2N},
\]
where
\[
u=C_1\left(
N^{-1/2}t^{1/2}+
T^{-1/12}t^{2/3}+
T^{-1/4}t\right).
\]
Here the ambient matrix dimension in the universality theorem is
$DN^L$, which accounts for the prefactor $DN^L$.

Next, $\mathcal P_A(G^{(N,T)})$ is a polynomial of degree at most $L$
in a Hayes' model with $VT$ GUE variables. Hence
\cite[Theorem~9.7]{CGVvH} yields
\[
\|\mathcal P_A(G^{(N,T)})\|
\leq
(1+\rho)\|\mathcal P_A(s^{(T)})\|
\]
except on an event of probability at most
\[
\frac{N^L}{c_H(L,V,T)\rho}
e^{-c_H(L,V,T)N\rho^2},
\]
provided
\[
D\leq
\exp\bigl(c_H(L,V,T)N\rho^2\bigr).
\]
Indeed, the constant $c$ in
\cite[Theorem~9.7]{CGVvH} satisfies
$c\geq c_H(L,V,T)$, and therefore
\[
c^{-1}\leq c_H(L,V,T)^{-1},
\qquad
e^{-cN\rho^2}
\leq
e^{-c_H(L,V,T)N\rho^2}.
\]

The limiting approximants satisfy the uniform norm estimate
\[
\sup_{T\geq 1}\max_{1\leq v\leq V}
\|s_v^{(T)}\|\leq C_2,
\]
by the uniform Haagerup inequality used in Appendix~B of \cite{CGVvH}. Therefore,
\[
\|\mathcal P_A(s^{(T)})\|\leq C_2M.
\]
Moreover, by the definition of \(\Delta_{L,V,M}(T)\),
\[
\|\mathcal P_A(s^{(T)})\|\leq
\|\mathcal P_A(x^{\Gamma})\|+\Delta_{L,V,M}(T).
\]
Combining these estimates gives
\begin{align*}
\|\mathcal P_A(G^{(N)})\|
\leq\|\mathcal P_A(G^{(N,T)})\|+u 
&\leq(1+\rho)\|\mathcal P_A(s^{(T)})\|+u \\
&\leq\|\mathcal P_A(x^{\Gamma})\|+\Delta_{L,V,M}(T)
+C_2M\rho+u.
\end{align*}
Taking the union bound over the two exceptional events proves
\eqref{eq:quantitative-tensor-gue-bound}.
\end{proof}
Denote the full tensor-GUE matrix dimension as $n_N=N^L$. Now choose
\begin{equation}\label{eq:large-coefficient-dimension}
D_N=N^{2L}=n_N^2.
\end{equation}
Then $D_N>n_N\geq N$
for $N\geq2$.

Fix \(\varepsilon,\eta\in(0,1)\), and put
\begin{equation}\label{eq:rho-choice}
\rho=\min\left\{1,\frac{\varepsilon}{6C_2M}\right\}.
\end{equation}
Define
\[
t_N=\log\left(\frac{12D_NN^L}{\eta}\right)
=\log\left(\frac{12}{\eta}\right)+3L\log N
\] 
and
\[
T_N=\left\lceil (1+t_N)^9\right\rceil.
\] 
\begin{corollary}
\label{cor:coefficient-dimension-larger-than-N}
For fixed \(L,V,M,\varepsilon,\eta\), there exists \(N_0\) such that,
for every \(N\geq N_0\), every $D\leq N^{2L}$,
and every self-adjoint linear pencil \(\mathcal P_A\) satisfying
\(\Lambda(A)\leq M\), we has
\[
\mathbb P\left[
\|\mathcal P_A(G^{(N)})\|>
\|\mathcal P_A(x^{\Gamma})\|+\varepsilon\right]
\leq \eta.
\]
\end{corollary}
\begin{proof}
As \(N\to\infty\),
\[
t_N=\mathcal{O}(\log N),\qquad
T_N=\mathcal{O}\bigl((\log N)^9\bigr).
\]
In particular,
\(T_N\leq e^N\)
for every sufficiently large $N$, so condition~(6.2) is
satisfied.
Therefore,
$N^{-1/2}t_N^{1/2}\rightarrow 0$,
while $T_N^{-1/12}t_N^{2/3}=
\mathcal O(t_N^{-1/12})\rightarrow 0$
and $T_N^{-1/4}t_N=
\mathcal O(t_N^{-5/4})\rightarrow 0$.
Moreover,
$\Delta_{L,V,M}(T_N)\rightarrow 0$
by Lemma~\ref{asytompzeroCLTmodulus}.
Recall that
$c_H(L,V,T_N)=
L^{-2}(C_0LVT_N)^{-2L}$.
Since \(T_N=\mathcal O((\log N)^9)\), we have
\[
c_H(L,V,T_N)N\rho^2
\asymp_{L,V,\rho}
\frac{N}{(\log N)^{18L}}.
\]
Consequently,
\[
\frac{\log D_N}{c_H(L,V,T_N)N\rho^2}
\asymp_{L,V,\rho}
\frac{(\log N)^{18L+1}}{N}
\longrightarrow 0.
\]
Thus, for all sufficiently large $N$, we get 
$D_N\leq
\exp\bigl(c_H(L,V,T_N)N\rho^2\bigr)$,
so that condition \eqref{eq:D-bound-cH} holds for every \(D\leq D_N\).
By the definition of \(t_N\), we have $2D_NN^Le^{-t_N}
=\frac{\eta}{6}$.
Moreover, $C_1VT_Ne^{-c_2N}\rightarrow0$.
Also,
\[
c_H(L,V,T_N)^{-1}
   =\mathcal O_{L,V}\big((\log N)^{18L}\big),
\]
and hence
\[
\log\left(
 \frac{N^L}{c_H(L,V,T_N)\rho}
\right)
=L\log N+\mathcal O_{L,V,\rho}(\log\log N).
\]
On the other hand,
\[
c_H(L,V,T_N)N\rho^2
   \asymp_{L,V,\rho}
   \frac{N}{(\log N)^{18L}}.
\]
Consequently,
\[
\frac{N^L}{c_H(L,V,T_N)\rho}
e^{-c_H(L,V,T_N)N\rho^2} \asymp_{L,V,\rho} e^{L \log(N) - \frac{N}{(\log N)^{18L}} + \order_{L,V,\rho}(\log \log N )}
\longrightarrow0.
\]
Finally, by \eqref{eq:rho-choice}, we have $C_2M\rho\leq \frac{\varepsilon}{6}$,
and hence, for sufficiently large \(N\),
\[
\Delta_{L,V,M}(T_N)
+C_2M\rho
+u(N,T_N,t_N)
\leq \varepsilon.
\]
Applying Theorem~\ref{thm:quantitative-tensor-gue-transfer} completes
the proof.
\end{proof}

\begin{remark}
Thus, the quantitative tensor-GUE method  allows $D_N=N^{2L}>N^L$.
More generally, every fixed polynomial growth regime $D_N=N^{aL}$,
$a>1$ is allowed.
\end{remark}

We have   the quantitative upper bound. In the following we prove that the lower bound is
uniform in the growing coefficient dimension by a finite-dimensional
operator-space argument.

\begin{lemma}[Uniform tensor-GUE lower bound]\label{lem:uniform-lower-bound}
Under the hypotheses of \cref{cor:coefficient-dimension-larger-than-N}, for every $\varepsilon>0$,
\[
\lim_{N\to\infty}
\sup_{\substack{D\leq N^{2L}\\
A_0,\ldots,A_V\in M_D(\mathbb C)_{\rm sa}\\
\Lambda(A)\leq M}}
\mathbb P\left[
 \|\mathcal P_A(G^{(N)})\|
 <
 \|\mathcal P_A(x^{\Gamma})\|-\varepsilon
\right]
=0.
\]
\end{lemma}

\begin{proof}
Put $\mathcal A=C^*(x^{\Gamma}_1,\ldots,x^{\Gamma}_V)$.
Each vertex algebra $C^*(x^{\Gamma}_v)\simeq C([-2,2])$
is nuclear, hence exact.  Since reduced graph products preserve
exactness \cite[Corollary~2.17]{CaspersFima}, the reduced graph-product
algebra $\mathcal A=C^*(x^{\Gamma}_1,\ldots,x^{\Gamma}_V)$
is exact.  Fix $\delta\in(0,1)$.
Applying the fixed-family exactness compression lemma 
\cref{lem:fixed-compression} to $1,x^{\Gamma}_1,\ldots,x^{\Gamma}_V\in\mathcal A$
gives an integer $m=m(L,V,\delta)$
with the following property.  For every $D\geq1$ and every
$A_0,\ldots,A_V\in M_D(\mathbb C)$, there are contractions $U,V\in M_{m,D}(\mathbb C)$
such that, on putting
\[
B_v=UA_vV^*\in M_m(\mathbb C),
\qquad 0\leq v\leq V,
\]
we have
\begin{equation}\label{Eqn=Tag612}
(1-\delta)\|\mathcal P_A(x^{\Gamma})\|
 \leq
 \left\|
 B_0\otimes1+\sum_{v=1}^V B_v\otimes s_v
 \right\|. 
\end{equation}
Here we have used the canonical isometric flip between the two
minimal tensor factors.  Notice also that
\[
\Lambda(B)
=\sum_{v=0}^V\|UA_vV^*\|
\leq\Lambda(A)
\leq M.
\]

For $B=(B_0,\ldots,B_V)$, write
\[
\mathcal Q_B(x)=B_0\otimes1+\sum_{v=1}^V B_v\otimes x_v,
\]
and define
\[
R_N
:=
\sup_{\substack{B_0,\ldots,B_V\in M_m(\mathbb C)\\
\Lambda(B)\leq M}}
\left|
 \|\mathcal Q_B(G^{(N)})\|-\|\mathcal Q_B(x^{\Gamma})\|
\right|.
\]
We claim that
\begin{equation}\label{Eqn=Tag613}
R_N\longrightarrow0
\quad\text{in probability.}
\end{equation}

Indeed, the coefficient set
\[
\mathcal K_m(M)
=
\left\{
(B_0,\ldots,B_V)\in M_m(\mathbb C)^{V+1}:
\Lambda(B)\leq M
\right\}
\]
is compact.  For each fixed $B\in\mathcal K_m(M)$,
\cite[Theorem~9.8]{CGVvH} gives
\[
\|\mathcal Q_B(G^{(N)})\|-\|\mathcal Q_B(x^{\Gamma})\|\longrightarrow0
\]
in probability, in fact almost surely.

Moreover, for $B,C\in\mathcal K_m(M)$,
\[
\begin{aligned}
\big|
 \|\mathcal Q_B(G^{(N)})\|-\|\mathcal Q_C(G^{(N)})\|
\big|
&\leq
\|B_0-C_0\|
 +\sum_{v=1}^V
   \|B_v-C_v\|\,\|G_v^{(N)}\|,\\
\big|
 \|\mathcal Q_B(x^{\Gamma})\|-\|\mathcal Q_C(x^{\Gamma})\|
\big|&\leq
\|B_0-C_0\|
 +\sum_{v=1}^V
   \|B_v-C_v\|\,\|x^{\Gamma}_v\|.
\end{aligned}
\]
By \cite[Theorem~9.8]{CGVvH}, applied to the coordinate
polynomials, the family $\max_{1\leq v\leq V}\|G_v^{(N)}\|$
is bounded in probability.  A finite-net argument on
$\mathcal K_m(M)$ therefore upgrades the pointwise convergence
to the uniform convergence  \eqref{Eqn=Tag613}.

On the random-matrix side,$\mathcal Q_B(G^{(N)})
 =
 (U\otimes1)\mathcal P_A(G^{(N)})(V^*\otimes1)$,
and consequently $\|\mathcal Q_B(G^{(N)})\|
\leq\|\mathcal P_A(G^{(N)})\|$.
Combining this with \eqref{Eqn=Tag612}  gives, uniformly over all
admissible $D$ and $A$,
\begin{equation} \label{Eqn=Tag614} 
\begin{aligned}
\|\mathcal P_A(G^{(N)})\|
\geq \|\mathcal Q_B(G^{(N)})\|
\geq \|\mathcal Q_B(x^{\Gamma})\|-R_N
\geq (1-\delta)\|\mathcal P_A(s)\|-R_N.
\end{aligned}
\end{equation}

Set
$C_s=\max\{1,\|x^{\Gamma}_1\|,\ldots,\|x^{\Gamma}_V\|\}$.
Then $\|\mathcal P_A(x^{\Gamma})\|\leq C_s\Lambda(A)\leq C_sM$.
Given $\varepsilon>0$, choose $0<\delta<\frac{\varepsilon}{2C_sM}$.
It follows from \eqref{Eqn=Tag613} and  \eqref{Eqn=Tag614}   that
\[
\begin{aligned}
\sup_{\substack{D\leq N^{2L}\\\Lambda(A)\leq M}}
\mathbb P\left[
 \|\mathcal P_A(G^{(N)})\| <
 \|\mathcal P_A(x^{\Gamma})\|-\varepsilon
\right]
\leq
\mathbb P\left[R_N>\frac{\varepsilon}{2}\right]
\longrightarrow0.
\end{aligned}
\]
This proves the lemma.
\end{proof}
\begin{lemma}[Finite polynomial spectral detector]\label{lem=finite-polynomial-spectral-detector}
For every $K,\varepsilon>0$, there are finitely many real
polynomials
\[
h_1,\ldots,h_J\in\mathbb R[t]
\]
such that the following holds.  If $X,Y$ are nonempty compact
subsets of $[-K,K]$ and $X\not\subseteq Y+[-\varepsilon,\varepsilon]$,
then, for some $j$,
\[
\sup_{x\in X}|h_j(x)|
>
2\sup_{y\in Y}|h_j(y)|. 
\]
The polynomials may moreover be chosen strictly positive on
$[-K,K]$.
\end{lemma}

\begin{proof}
Choose a finite $\varepsilon/8$-net
$z_1,\ldots,z_J$ in $[-K,K]$.  For each $j$, choose a continuous
function $f_j:[-K,K]\to[0,1]$ such that
\[
f_j(t)=1\quad\text{if }|t-z_j|\leq\varepsilon/8,
\qquad
f_j(t)=0\quad\text{if }|t-z_j|\geq\varepsilon/2.
\]
By the Weierstrass approximation theorem, choose
$h_j\in\mathbb R[t]$ satisfying
\[
\left\|
 h_j-\left(\frac18+f_j\right)
\right\|_{C([-K,K])}
<\frac1{64}.
\]
Then $h_j\geq7/64$ on $[-K,K]$.

If $x\in X$ satisfies $\mathrm{dist}(x,Y)>\varepsilon$,
choose $j$ with $|x-z_j|\leq\varepsilon/8$.  We then have
$h_j(x)\geq\frac{71}{64}$.
For every $y\in Y$,
\[
|y-z_j|
\geq |y-x|-|x-z_j|
>\frac{7\varepsilon}{8},
\]
so $f_j(y)=0$ and
$
|h_j(y)|\leq\frac9{64}$.
This proves (6.15).
\end{proof}
Combining Corollary~\ref{cor:coefficient-dimension-larger-than-N} and
Lemma~\ref{lem:uniform-lower-bound} gives the full norm convergence. Note that \eqref{Eqn=618} must be interpreted as a spectral gap condition.

\begin{theorem}[Growing-coefficient tensor-GUE convergence]\label{thm:tensor-gue-D-larger-than-matrix-dimension}
Fix $L,V,M$.  Then the following statements hold uniformly over
all $D\leq N^{2L}$ and all self-adjoint pencils $\mathcal P_A$ satisfying
$\Lambda(A)\leq M$.

\begin{enumerate}
\item For every $\varepsilon>0$,
\begin{equation}\label{Eqn=Tag616}
\mathbb P\left[\left|
  \|\mathcal P_A(G^{(N)})\|-\|\mathcal P_A(x^{\Gamma})\|
 \right|>\varepsilon
\right]
\longrightarrow0.
\end{equation}

\item For every $\varepsilon>0$,
\begin{equation}\label{Eqn=Tag617}
\mathbb P\left[
 \mathrm{sp}(\mathcal P_A(G^{(N)}))
 \not\subseteq
 \mathrm{sp}(\mathcal P_A(x^{\Gamma}))+[-\varepsilon,\varepsilon]
\right]
\longrightarrow0.
\end{equation}
\end{enumerate}

Consequently, for every $\gamma>0$,
\begin{equation}\label{Eqn=618}
\sup_{\substack{D\leq N^{2L},\,\Lambda(A)\leq M\\
\mathrm{dist}(0,\mathrm{sp}(\mathcal P_A(x^{\Gamma})))\geq\gamma}}
\mathbb P\left[
\mathcal  P_A(G^{(N)})\textrm{ is not invertible, or }
 \|\mathcal P_A(G^{(N)})^{-1}\|>\frac2\gamma
\right]
\longrightarrow0. 
\end{equation}
\end{theorem}

\begin{proof}
Statement \eqref{Eqn=Tag616} follows from Corollary \ref{cor:coefficient-dimension-larger-than-N} and Lemma \ref{lem:uniform-lower-bound}.
It remains to prove \eqref{Eqn=Tag617}.

Put
\[
a_L=2^L-2,
\qquad
C_s=\max\{1,\|x^{\Gamma}_1\|,\ldots,\|x^{\Gamma}_V\|\},
\qquad
K_0=(C_s+a_L)M.
\]
By (6.2''),
\[
\|\mathcal P_A(s^{(T)})\|\leq K_0
\qquad(T\geq1).
\tag{6.19}
\]

Fix $\varepsilon>0$.  Apply Lemma~\ref{lem=finite-polynomial-spectral-detector} with
\[
K=2K_0,
\qquad
\varepsilon \textrm{ replaced by } \frac{\varepsilon}{3},
\]
and let $\mathcal C$ denote the resulting finite family of
polynomials.  Adjoin to $\mathcal C$ the polynomial
\(
g(t)=1+\frac{t^2}{K_0^2}
\).
Let $q=\max_{h\in\mathcal C\cup\{g\}}\deg h$.
For $N\geq2$, set
\[
t_N=(3L+3)\log N,
\qquad
T_N=\left\lceil(1+t_N)^9\right\rceil,
\]
and define
\[
c_{H,q}(L,V,T)
=
(Lq)^{-2}(C_0LVT)^{-2Lq}.
\tag{6.20}
\]
If $h\in\mathcal C\cup\{g\}$, then
$h(\mathcal P_A(G^{(N,T)}))$
is a polynomial of degree at most $Lq$ in the Hayes variables.
Therefore \cite[Theorem~9.7]{CGVvH}, with its parameter
$\rho=1$, gives
\[
\begin{aligned}
\mathbb P\left[
 \|h(\mathcal{P}_A(G^{(N,T_N)}))\|
 \geq
 2\|h(\mathcal{P}_A(s^{(T_N)}))\|
\right]\leq
\frac{N^L}{c_{H,q}(L,V,T_N)}
e^{-c_{H,q}(L,V,T_N)N},
\end{aligned}
\tag{6.21}
\]
provided $D\leq
\exp\big(c_{H,q}(L,V,T_N)N\big)$.
This condition holds for every $D\leq N^{2L}$ and all
sufficiently large $N$, since
\[
c_{H,q}(L,V,T_N)^{-1}
=
\order \big((\log N)^{18Lq}\big).
\]

Taking a union bound over the finite set
$\mathcal C\cup\{g\}$, the total probability in (6.21) tends to
zero.  On the complementary event, (6.19) gives $\| \mathcal{P}_A(s^{(T_N)}))\|\leq K_0$,
and hence
$\| \mathcal{P}_A(G^{(N,T_N)}) \|< 2 K_0$.
Since $\mathcal P_A(G^{(N,T_N)})$ is self-adjoint,
\[
1+
\frac{\|\mathcal P_A(G^{(N,T_N)})\|^2}{K_0^2}
< 5,
\]
so
\[
\mathrm{sp}(\mathcal P_A(G^{(N,T_N)}))
\subseteq[-2K_0,2K_0].
\tag{6.22}
\]

Suppose now that
\[
\mathrm{sp}(\mathcal P_A(G^{(N,T_N)}))
\not\subseteq
\mathrm{sp}(\mathcal P_A(s^{(T_N)}))
 +\left[-\frac{\varepsilon}{3},
         \frac{\varepsilon}{3}\right].
\]
By (6.19), (6.22), and the spectral detector Lemma \ref{lem=finite-polynomial-spectral-detector}, there is
$h\in\mathcal H$ such that
\[
\|h(\mathcal P_A(G^{(N,T_N)}))\|
>
2\|h(\mathcal P_A(s^{(T_N)}))\|,
\]
contradicting the complementary event in (6.21).  Therefore
\[
\mathrm{sp}(\mathcal P_A(G^{(N,T_N)}))
\subseteq
\mathrm{sp}(\mathcal P_A(s^{(T_N)}))
 +\left[-\frac{\varepsilon}{3},
         \frac{\varepsilon}{3}\right]
\tag{6.23}
\]
outside an event whose probability tends to zero uniformly.

The universality comparison gives
\[
\mathrm{sp}(\mathcal P_A(G^{(N)}))
\subseteq
\mathrm{sp}(\mathcal P_A(G^{(N,T_N)}))
 +[-u_N,u_N]
\tag{6.24}
\]
outside an event of probability at most $2DN^Le^{-t_N}+C_1VT_Ne^{-c_2N}$,
where
\[
u_N=
C_1\left(
 N^{-1/2}t_N^{1/2}
 +T_N^{-1/12}t_N^{2/3}
 +T_N^{-1/4}t_N
\right)
\longrightarrow0.
\]
Because $D\leq N^{2L}$, $2DN^Le^{-t_N}\leq2N^{-3}$,
so this exceptional probability also tends to zero.

Finally, (6.2'') gives
\[
\mathrm{sp}(\mathcal P_A(s^{(T_N)}))
\subseteq
\mathrm{sp}(\mathcal P_A(x^{\Gamma}))
+
\left[
 -\frac{a_LM}{\sqrt{T_N}},
 \frac{a_LM}{\sqrt{T_N}}
\right].
\tag{6.25}
\]
Combining (6.23)--(6.25), and using
\(
u_N+\frac{a_LM}{\sqrt{T_N}}
<\frac{2\varepsilon}{3}
\)
for all sufficiently large $N$, proves (6.17).

If $\mathrm{dist}
 \bigl(0,\mathrm{sp}(\mathcal P_A(x^{\Gamma}))\bigr)
> \gamma$,
apply \eqref{Eqn=Tag617} with $\varepsilon=\gamma/2$.  Then, with probability
tending uniformly to one, $\mathrm{dist}
 \bigl(0,\mathrm{sp}(P_A(G^{(N)}))\bigr)
\geq\frac{\gamma}{2}$.
This proves \eqref{Eqn=618}. If $\mathrm{dist}
 \bigl(0,\mathrm{sp}(\mathcal P_A(x^{\Gamma}))\bigr)
= 0 $ then the statement follows from \eqref{Eqn=Tag617} as well. 
\end{proof}

\begin{corollary}[Bounded-degree polynomial tensor-GUE convergence]\label{Bound-degree-tensor-gue-strongconvergence}
Fix $L,V,d,M$. There is an integer
\( c_{V,d}\geq1\)
such that, for every $\varepsilon>0$,
\[
\begin{aligned}
\lim_{N\to\infty}
\sup_{\substack{
 D\geq1,\;c_{V,d}D\leq N^{2L}\\
 P=P^*,\;\deg P\leq d\\
 \Lambda_d(P)\leq M
}}
\mathbb P\left[
 \left|
  \|P(G^{(N)})\|-\|P(x^{\Gamma})\|
 \right|>\varepsilon
\right]
=0.
\end{aligned}
\tag{6.26}
\]
Moreover,
\[
\begin{aligned}
\lim_{N\to\infty}
\sup_{\substack{
 D\geq1,\;c_{V,d}D\leq N^{2L}\\
 P=P^*,\;\deg P\leq d\\
 \Lambda_d(P)\leq M
}}
\mathbb P\left[
 \mathrm{sp}(P(G^{(N)}))
 \not\subseteq
 \mathrm{sp}(P(x^{\Gamma}))+[-\varepsilon,\varepsilon]
\right]
=0.
\end{aligned}
\tag{6.27}
\]
For a non-self-adjoint polynomial, the norm-convergence
statement (6.26) remains valid after replacing
$c_{V,d}$ by $2c_{V,d}$ and applying the preceding
argument to the Hermitian dilation
\[
\mathcal D(P)
=
\begin{pmatrix}
0&P\\
P^*&0
\end{pmatrix},
\qquad
\|\mathcal D(P)\|=\|P\|.
\]
We do not claim any spectral-inclusion statement  for the generally
nonnormal operator $P(G^{(N)})$.
\end{corollary}

\begin{proof}
We use the self-adjoint linearization trick; see \cite[Sections~2 and~7]{HT05},
\cite[Section~2.6]{And13}, and
\cite[Section~3, Step~1]{Male2012}.

For fixed $V$ and $d$, the standard  self-adjoint
Schur-complement linearization gives constants
$c_{V,d}$ and $C_{V,d}$ with the following properties. 

For every self-adjoint polynomial $P$ of degree at most $d$ and
every real $\lambda$, there is a self-adjoint linear pencil
\[
  L_{P,\lambda}(x)
  =
  B_0(\lambda)\otimes1
  +
  \sum_{v=1}^V B_v\otimes x_v
\]
such that:

\begin{enumerate}
\item the coefficient dimension of $L_{P,\lambda}$ is at most
      $c_{V,d}D$;

\item
\begin{equation}\label{invertible}
    \lambda-P(x)\text{ is invertible}
  \quad\Longleftrightarrow\quad
  L_{P,\lambda}(x)\text{ is invertible};
\end{equation}

\item for $\lambda$ in a fixed bounded interval,
\begin{equation}\label{Bv}
    \sum_{v=0}^V\|B_v\|
  \leq
  C_{V,d}(1+M+|\lambda|);
\end{equation}

\item for every $R,\varepsilon>0$ there is
\[
  C=C(V,d,M,R,\varepsilon)
\]
such that, whenever $\max_v \vert x_v \vert \leq R$ and
\[
  \mathrm{dist}
  \bigl(
    \lambda,\mathrm{sp}( P(x))
  \bigr)
  \geq\varepsilon,
\]
one has
\begin{equation}\label{LpC}
    \|L_{P,\lambda}(x)^{-1}\|\leq C.
\end{equation}

\end{enumerate}
Moreover,
\(
L_{P,\lambda}(x)-L_{P,\mu}(x)
=(\lambda-\mu)E\),
where
\(\|E\|=1\).
The last assertion follows directly from the block inverse
formula for the Schur-complement linearization: the auxiliary
block is unitriangular, while the only resolvent term is
$(\lambda-P(x))^{-1}$.

Choose $R>
  \max_{1\leq v\leq V}\|s_v\|+1$.
By \cref{thm:tensor-gue-D-larger-than-matrix-dimension} applied to the coordinate linear pencils,
\[
  \mathbb P\left[
   \max_v\|G_v^{(N)}\|>R
  \right]\longrightarrow0. 
\]
On the complementary event,
\[
\|P(G^{(N)})\|,
\|P(x^{\Gamma})\|
\leq
K:=M\max\{1,R\}^{d}.
\]
Thus all relevant spectra are contained in the fixed interval
\(
I=[-K,K]\).
Moreover, for every $\lambda\in I$, \eqref{Bv} gives
\[
\sum_{v=0}^V\|B_v(\lambda)\|
\leq
M':=C_{V,d}(1+M+K).
\]
The constants $K$ and $M'$ are independent of $D$, $P$, and
$\lambda$.

Fix $\varepsilon>0$.
By \eqref{LpC}, there exists $C=C(V,d,M,R,\varepsilon/2)$
such that
\(
\|L_{P,\lambda}(x^{\Gamma})^{-1}\|\le C
\)
whenever
\(
\mathrm{dist}
(\lambda,\mathrm{sp}(P(x^{\Gamma})))
\ge\frac{\varepsilon}{2}.
\)

Set $\gamma=C^{-1}$.
Choose
\(
0<\delta<
\min\left\{
\frac{\varepsilon}{2},
\frac{\gamma}{4}
\right\},
\)
and let
\(
\lambda_1,\ldots,\lambda_m
\)
be a finite $\delta$-net of the bounded interval containing all
relevant spectra.
For every net point satisfying
\(
\mathrm{dist}
(\lambda_j,\mathrm{sp}(P(x^{\Gamma})))
\ge
\frac{\varepsilon}{2},
\)
we have
\(
\mathrm{dist}\left(
0,\mathrm{sp}(L_{P,\lambda_j}(x^{\Gamma}))
\right)
\ge\gamma.
\)

Applying (6.18) to the finitely many linear pencils
\(
L_{P,\lambda_j},
\)
and taking a union bound, we obtain that, with probability
tending uniformly to one,
\[
\|L_{P,\lambda_j}(G^{(N)})^{-1}\|
\le\frac{2}{\gamma}
\]
for every such net point.

Now let
\(\lambda\)
satisfy
\(
\mathrm{dist}
(\lambda,\mathrm{sp}(P(x^{\Gamma})))
\ge\varepsilon,
\)
and choose
\(\lambda_j\)
with
\(|\lambda-\lambda_j|<\delta.
\)
Then
\(
\mathrm{dist}
(\lambda_j,\mathrm{sp}(P(x^{\Gamma})))
\ge\frac{\varepsilon}{2},
\)
and
\(
L_{P,\lambda}(G^{(N)})
=L_{P,\lambda_j}(G^{(N)})
+(\lambda-\lambda_j)E,
\)
where
\(\|E\|=1\).
Therefore
\[
\left\|
L_{P,\lambda_j}(G^{(N)})^{-1}
\right\| \left\|
\bigl(
L_{P,\lambda}(G^{(N)})
-
L_{P,\lambda_j}(G^{(N)})
\bigr)
\right\|
<
\frac12,
\]
because
\(
\delta<\frac{\gamma}{4}.
\)
Hence a standard von Neumann series argument (see \cite[Theorem 1.2.3]{Murphy}) shows that
\(L_{P,\lambda}(G^{(N)})\)
is invertible.

Using the linearization equivalence (6.28), we conclude that
\[
\mathrm{sp}(P(G^{(N)}))
\subseteq
\mathrm{sp}(P(x^{\Gamma}))
+
[-\varepsilon,\varepsilon]
\]
with probability tending uniformly to one. This proves (6.27),
and therefore the upper estimate in (6.26).

For the lower norm bound, apply the fixed-target exactness
compression argument of \cref{lem:uniform-lower-bound} to the finite family
\(\{w(x^{\Gamma}):|w|\leq d\}\).
The graph-product algebra $C^*(x^{\Gamma}_1,\ldots,x^{\Gamma}_V)$ is exact.
Consequently, arbitrary coefficient matrices may be compressed
to a fixed matrix dimension while preserving
$\|P(x^{\Gamma})\|$ from below. At that fixed dimension,
the polynomial tensor-GUE strong convergence theorem
\cite[Theorem~9.8]{CGVvH}, together with compactness of the
coefficient ball, gives uniform convergence. This proves the
lower bound and completes (6.26).

For a non-self-adjoint polynomial use the Hermitian dilation
\[
  \mathcal D(P)
  =
  \begin{pmatrix}
    0&P\\
    P^*&0
  \end{pmatrix},
  \qquad
  \|\mathcal D(P)\|=\|P\|.
\]
\end{proof}
\subsection{Growing coefficient  strong convergence of the Erd\H{o}s--R\'enyi
$q$-Gaussian matrix model}

We now combine the graph-product approximation of \cref{sec:multistrongconvergence} with the
growing-coefficient tensor-GUE approximation established above. For a
fixed finite graph, the underlying tensor-GUE strong convergence is
\cite[Theorem~9.8]{CGVvH}. In the present setting, however, the
interaction graph has $rk$ vertices and varies with $k$, while the
allowable coefficient dimension grows with, and is eventually chosen
to exceed, the dimension of the matrix models.

We therefore fix $k$, condition on a realization of the interaction
graph, realize that graph as the disjointness graph of finitely many
tensor supports, and apply Corollary~6.10 at a suitably chosen local
dimension $N_k$. Since there are only finitely many graphs on $V_k$,
the choice of $N_k$ can be made uniformly over all graph realizations.
A diagonal choice of $N_k$ then combines this matrix approximation
with \cref{Thm=StrongConvergence}.

Fix $r\geq 1$ and $-1<q<1$, and assume that the canonical map
\(
\Phi_q:\mathcal T^u_{q,r}\longrightarrow \mathcal T^{\mathrm{red}}_{q,r}
\)
is an isomorphism. In particular, the conclusions below hold whenever
$|q|<\sqrt2-1$ \cite{PuszWoronowicz}.

Let
\[
V_k=[k]\times[r],
\qquad
m_k=|V_k|=rk.
\]
We introduce a tensor-coordinate set consisting of one private
coordinate for each vertex and one pair coordinate for each
unordered pair of distinct vertices:
\[
I_k
=
\{p_v:v\in V_k\}
\sqcup
\{p_{\{u,v\}}:\{u,v\}\in\binom{V_k}{2}\};
\]
so the $p_v$ and $p_{\{ u,v \}}$ are by definition the labels of the tensor coordinates.  
And we denote
\[
L_k:=|I_k|
=
m_k+\binom{m_k}{2}.
\]
For a graph $\Gamma$ with vertices $V_k$ and $v\in V_k$, set
\[
K_v^\Gamma
 =\{p_v\}\cup
 \bigl\{\{u,v\}:u\neq v,\ u\not\sim_\Gamma v\bigr\}.
\]
Then
\[
K_u^\Gamma\cap K_v^\Gamma=\varnothing
\quad\Longleftrightarrow\quad
u\sim_\Gamma v.
\] 
For $0 \leq q < 1$, let $d_k=1$. For $ -1< q < 0$, let $\rho_k:\mathrm{Cliff}(V_k)
\longrightarrow M_{d_k}(\mathbb C)$
be a faithful finite-dimensional representation that preserves
the canonical Clifford trace, for example, the GNS
representation of the canonical trace.  We put
$c_v^{(k)}=\rho_k(c_v)$.

For $N\geq1$, let $G_v^{(N,\Gamma)}$ be the
tensor-GUE variable associated with $K_v^\Gamma\subseteq I_k$, and put
\[
\widehat G_v^{(N,\Gamma)}=
\begin{cases}
G_v^{(N,\Gamma)},& 0 \leq q < 1,\\
 c_v^{(k)}\otimes G_v^{(N,\Gamma)},& -1 < q <0.
\end{cases}
\]
Finally, define
\begin{equation}\label{eq:q-gaussian-matrix-model}
X_{k,a}^{(N,\Gamma)}
 =\frac1{\sqrt{k}}\sum_{i=1}^k
   \widehat G_{(i,a)}^{(N,\Gamma)},
\qquad a\in[r].
\end{equation}
The dimension of matrices in \eqref{eq:q-gaussian-matrix-model} is
\[
n_{k,N}=d_kN^{L_k}.
\]
For $D,d\geq1$, write
\[
\mathscr P_d(D)=
\left\{
P\in M_D(\mathbb C)\otimes\mathbb C\langle x_1,\ldots,x_r\rangle:
\deg P\leq d,\ \Lambda_d(P)\leq1
\right\}.
\]

\begin{theorem}[Strongly convergent q-Gaussian matrix model]
\label{thm:q-gaussian-complete-matrix-model}
Fix $r,d\geq1$ and assume that $\Phi_q$ is an isomorphism. In particular, this holds for $|q|<\sqrt{2}-1$. There exist
integers $N_k\to\infty$ and $D_k\to\infty$ satisfying $D_k>n_{k}$, such that, with
\[
X^{(k)}=
\bigl(X_{k,1}^{(N_k,\Gamma_k)},\ldots,
      X_{k,r}^{(N_k,\Gamma_k)}\bigr),
\qquad
n_k=d_kN_k^{L_k},
\]
the following holds:

\begin{enumerate}

\item For every $\varepsilon>0$,
\begin{equation}\label{eq:uniform-additive-q-model}
\sup_{1\leq D\leq D_k}
\sup_{P\in\mathscr P_d(D)}
\mathbb P_{\Gamma,\mathrm{GUE}}
\left(
\left|\|P(X^{(k)})\|-\|P(s^{(q)})\|\right|>\varepsilon
\right)
\longrightarrow0.
\end{equation}

\item For every $\varepsilon>0$,
\begin{equation}\label{eq:uniform-relative-q-model}
\sup_{1\leq D\leq D_k}
\sup_{\substack{0\neq P\in M_D\otimes
\mathbb C\langle x_1,\ldots,x_r\rangle\\
\deg P\leq d}}
\mathbb P_{\Gamma,\mathrm{GUE}}
\left(
\left|
\frac{\|P(X^{(k)})\|}{\|P(s^{(q)})\|}-1
\right|>\varepsilon
\right)
\longrightarrow0.
\end{equation}

\item The weak converge: for every scalar
noncommutative polynomial $Q$,
\begin{equation}\label{eq:trace-convergence-q-model}
\operatorname{tr}_{n_k}\bigl(Q(X^{(k)})\bigr)
\longrightarrow
\tau_q\bigl(Q(s^{(q)})\bigr)
\end{equation}
in probability.
\end{enumerate}
\end{theorem}

\begin{proof}
Let
\(
S_k^{(q,\Gamma)}=
\bigl(S_{k,1}^{(q,\Gamma)},\ldots,S_{k,r}^{(q,\Gamma)}\bigr)
\)
be the graph-product or Clifford-twisted graph-product tuple from
\cref{sec:multistrongconvergence}. By \cref{Thm=StrongConvergence} and the bound
\[
\|P(s^{(q)})\|
 \leq \Lambda_d(P)
 \max\left\{1,\frac{2}{\sqrt{1-q}}\right\}^{d},
\]
there are positive numbers $\varepsilon_k\rightarrow0$ and $\beta_k\rightarrow0$
such that
\begin{equation}\label{eq:graph-approximation-stage}
\sup_{D\geq1}\sup_{P\in\mathscr P_d(D)}
\mathbb P_\Gamma
\left(
\left|
\|P(S_k^{(q,\Gamma_k)})\|-\|P(s^{(q)})\|
\right|>\varepsilon_k
\right)
\leq\beta_k.
\end{equation}

For a graph $\Gamma$ on $V_k$ and
$P\in\mathscr P_d(D)$, substitute the averages in
\eqref{eq:q-gaussian-matrix-model} and denote the resulting polynomial
in the $m_k$ tensor-GUE variables by $\widetilde P_\Gamma$. More precisely, put

\[
\widetilde P_\Gamma
(z_{(i,a)})
=
P\!\left(
\frac1{\sqrt k}
\sum_{i=1}^k
\widehat z_{(i,1)},
\;
\dots,
\;
\frac1{\sqrt k}
\sum_{i=1}^k
\widehat z_{(i,r)}
\right),
\]
where
\[
\widehat z_{(i,a)}
=
\begin{cases}
z_{(i,a)},
&
0 \leq q < 1,
\\[1ex]
c^{(k)}_{(i,a)}
\otimes
z_{(i,a)},
&
-1 < q < 0.
\end{cases}
\]
Then
\[ 
P\bigl(X_k^{(N,\Gamma)}\bigr)
 =\widetilde P_\Gamma\bigl(G^{(N,\Gamma)}\bigr),
\qquad
P\bigl(S_k^{(q,\Gamma)}\bigr)
 =\widetilde P_\Gamma(x^\Gamma),
\]
where, for $q<0$, the Clifford matrices are absorbed into the matrix
coefficients of $\widetilde P_\Gamma$. Then \(
\deg(\widetilde P_\Gamma)\le d 
\), the coefficient dimension is \(d_kD\), and
\(\Lambda_d(\widetilde P_\Gamma)
\le k^{d/2}\). Indeed, every monomial of degree $m$
produces at most $k^m$ summands,
each carrying the coefficient
$k^{-m/2}$,
while every Clifford product has norm one.
Therefore \cref{Bound-degree-tensor-gue-strongconvergence} applies directly to
$\widetilde P_\Gamma$
with
\(L=L_k,
V=m_k=rk,
M=k^{d/2}\).

Let \(
c_k:=c_{m_k,d}\)
denote the corresponding linearization constant (see \cref{Bound-degree-tensor-gue-strongconvergence}).
It is independent of $\Gamma$, $D$, and the
particular polynomial in $\mathscr P_d(D)$.
Choose $N_k$ sufficiently large that
\begin{equation}\label{eq;size-choice}
    N_k^{L_k}>2c_kd_k^2+2.
\end{equation}
and
\begin{equation}\label{eq:tensor-transfer-stage}
\sup_{\Gamma}
\sup_{\substack{D\geq1\\
c_kd_kD\leq N_k^{2L_k}}}
\sup_{P\in\mathscr P_d(D)}
\mathbb P_{\mathrm{GUE}}
\left(
\left|
\|P(X_k^{(N_k,\Gamma)})\|
 -\|P(S_k^{(q,\Gamma)})\|
\right|>\varepsilon_k
\right)
\leq k^{-1}.
\end{equation}
We enlarge $N_k$, if necessary, so that the moment-convergence part of
the tensor-GUE theorem also gives, for every graph $\Gamma$ on $V_k$
and every word $w$ in $r$ letters of length at most $k$,
\begin{equation}\label{eq:tensor-moment-stage}
\mathbb P_{\mathrm{GUE}}
\left(
\left|
\operatorname{tr}_{n_{k,N_k}}
 \bigl(w(X_k^{(N_k,\Gamma)})\bigr)
 -\varphi_{q,k}
 \bigl(w(S_k^{(q,\Gamma)})\bigr)
\right|>k^{-1}
\right)
\leq k^{-1}.
\end{equation}
This simultaneous choice is possible, because for fixed $k$, only
finitely many graphs and finitely many words occur. In the case
$q<0$, trace-preserving $\rho_k$ identifies the limiting trace
with the Clifford-twisted graph-product vacuum state.

Now define
\[ 
D_k=
\left\lfloor
\frac{N_k^{2L_k}}{c_kd_k}
\right\rfloor.
\]
Writing $x_k=N_k^{L_k}$, \eqref{eq;size-choice} gives
\[
D_k
\geq \frac{x_k^2}{c_k d_k}-1
>2d_kx_k-1.
\]
Since $n_k=d_kx_k>0$, it follows that
$D_k>n_k$ as in the requirements of the theorem.

 By the triangle
inequality, conditioning on $\Gamma_k$, 
\eqref{eq:graph-approximation-stage} and \eqref{eq:tensor-transfer-stage} yield,
\begin{equation}\label{Eqn=Tag642}
\sup_{1\leq D\leq D_k}
\sup_{P\in\mathscr P_d(D)}
\mathbb P_{\Gamma,\mathrm{GUE}}
\left(
\left|
\|P(X^{(k)})\|-\|P(s^{(q)})\|
\right|>2\varepsilon_k
\right) \leq k^{-1}+\beta_k\longrightarrow0.
\end{equation}
This proves \eqref{eq:uniform-additive-q-model}.

Finally, since the monomials of degree at most $d$ in the
$q$-Gaussians are linearly independent (see Remark \ref{Rmk=LinIn}), their coordinate
functionals are completely bounded. Hence there is a constant
$C_{q,r,d}$ such that, at every matrix coefficient dimension,
\[
  \Lambda_d(P)
  \leq
  C_{q,r,d}\|P(s^{(q)})\|.
\]
Scaling  \eqref{Eqn=Tag642}  therefore proves
\eqref{eq:uniform-relative-q-model}.

Finally, fix a word $w$. For all sufficiently large $k$, $|w|\leq k$.
By \eqref{eq:tensor-moment-stage},
\[
\operatorname{tr}_{n_k}\bigl(w(X^{(k)})\bigr)
-
\varphi_{q,k}
 \bigl(w(S_k^{(q,\Gamma_k)})\bigr)
\longrightarrow0
\]
in probability. On the other hand, \cref{Thm=ConvergenceMomentsToeplitz} gives
\[
\varphi_{q,k}
 \bigl(w(S_k^{(q,\Gamma_k)})\bigr)
\longrightarrow
\tau_q\bigl(w(s^{(q)})\bigr)
\]
in probability. Hence \eqref{eq:trace-convergence-q-model} holds for
words, and therefore for all scalar polynomials by linearity.
\end{proof}

\section{Applications}\label{sec:MF}
\begin{theorem}[MF property]\label{thm:q-gaussian-mf}
Let $r\geq 1$ and let $|q|<\sqrt{2}-1$.
Then $A_q(\mathbb C^r)$ is an MF algebra; equivalently, there
exist positive integers $n_m$ and a unital injective $*$-homomorphism
\[
A_q(\mathbb C^r)
\longrightarrow
\prod_{m=1}^{\infty}M_{n_m}(\mathbb C)\bigg/
\bigoplus_{m=1}^{\infty}M_{n_m}(\mathbb C),
\]
where
\[
\bigoplus_{m=1}^{\infty}M_{n_m}(\mathbb C)
=\left\{(a_m)_m:
\lim_{m\to\infty}\|a_m\|=0
\right\}.
\]
\end{theorem}

\begin{proof}
Let
\(
\mathcal P_0=
\mathbb Q(i)\langle x_1,\ldots,x_r\rangle
\)
be the countable $*$-algebra of noncommutative polynomials with
coefficients in $\mathbb Q(i)$, where the variables are seen
to be self-adjoint. 

Enumerate
\(
\mathcal P_0=\{P_1,P_2,\ldots\}
\). By the strongly convergent matrix model for $q$-Gaussian with $|q|< \sqrt{2}-1$, for any fixed noncommutative polynomial there are random self-adjoint matrix tuples \( X^{(k)}=(X^{(k)}_1,\cdots, X^{(k)}_r)\in M_{N_k}(\mathbb{C})^r_{sa}\) such that
\[\Vert P(X^{(k)})\Vert\longrightarrow \Vert P(s^{(q)}_1,\cdots,s^{(q)}_r)\Vert\]
in probability.

After passing to an increasing subsequence of matrix model parameters, for every $m$ we may choose a tuple of deterministic matrices
\[Y^{(m)}=(Y^{(m)}_1,\cdots,Y^{(m)}_r)\in M_{n_m}(\mathbb{C})_{sa}^r\]
such that
\[\max_{1\leq j\leq m}\bigg|\Vert P_j(Y^{(m)})\Vert -\Vert P_j(s^{(q)}_1,\cdots, s^{(q)}_r)\Vert \bigg|\leq \frac{1}{m}.\]
Indeed, convergence in probability and the union bound show that the corresponding event has positive probability once the parameter of the matrix model is sufficiently large.

Consequently,  for every $P\in \mathcal P_0$,
\begin{equation}\label{MF1}
    \lim_{m\to \infty}\Vert P_j(Y^{(m)}_1,\cdots,Y^{(m)}_r)\Vert\longrightarrow \Vert P_j(s^{(q)}_1,\cdots, s^{(q)}_r)\Vert.
\end{equation}

Put \[\mathcal{Q}:= \prod_{m=1}^{\infty}M_{n_m}(\mathbb C)\bigg/
\bigoplus_{m=1}^{\infty}M_{n_m}(\mathbb C)\]
and denote the quotient map by
\[\pi:\prod_{m=1}^{\infty}M_{n_m}(\mathbb C)\longrightarrow \mathcal Q\]
Let $\mathcal P_0(s^{(q)})$ be the polynomial algebra generated by the $q$-Gaussian system, and define $\Psi_0:\mathcal P_0(s^{(q)})\to \mathcal Q$ by 

\[\Psi_0( P(s_1^{(q)},\cdots, s_r^{(q)}))=\pi\big((P(Y_1^{(m)},\cdots,Y_r^{(m)}))_{m=1}^{\infty}\big).\]
This map is well-defined. Indeed, if \(P(s_1^{(q)},\cdots, s_r^{(q)})=0\), then \eqref{MF1} gives that \[\lim_{m\to \infty}\Vert P_j(Y^{(m)}_1,\cdots,Y^{(m)}_r)\Vert\rightarrow 0,\]
so the corresponding sequence belongs to $\bigoplus_{m=1}^{\infty}M_{n_m}(\mathbb C)$; in fact $P=0$ by Remark \ref{Rmk=LinIn}.

Since the norm in $\mathcal Q$ is a quotient norm, we have \[\Vert \Psi_0(P(s^{(q)}))\Vert =\limsup_{m\to \infty }\Vert P(Y^{(m)})\Vert=\Vert P(s^{(q)})\Vert. \]
Thus $\Psi_0$ is a isomerty. Since $\mathcal P_0(s^{(q)})$ is norm dense in $A_q(\mathbb C^r)$, $\Psi_0$ extends uniquely to a unital injective $*$-homomorphism $\Psi:A_q(\mathbb C^r)\to \mathcal Q$. Hence $A_q(\mathbb C^r)$ is MF algebra.
\end{proof}

\begin{theorem}[The extension semigroup]\label{thm:q-gaussian-ext}
Let $r\geq2$ and let $|q|<\sqrt{2}-1$.
Then the Brown--Douglas--Fillmore extension semigroup
\(\mathrm{Ext}\bigl(A_q(\mathbb C^r)\bigr)
\)
is not a group.
\end{theorem}

\begin{proof}
    The Brown criterion states  that if a unital separable $C^*$-algebra is MF but not quasi-diagonal, then its Brown-Douglas-Fillmore extension semigroup is not a group; see \cite{Brown2004}. Therefore we only need to show that $A_q(\mathbb C^r)$ is not quasidiagonal. 

    Suppose, for contradiction, that $A_q(\mathbb C^r)$ is quasidiagonal. Then by the caraterization of quasidiagonal, there exist unital completely positive maps $\phi_n:A_q(\mathbb C^r)\to M_{k_n}(\mathbb C)$ such that, for $a,b\in A_q(\mathbb C^r) $,
    \begin{equation}\label{Ext1}
        \Vert \phi_n(ab)-\phi_n(a)\phi_n(b)\Vert\longrightarrow 0;
    \end{equation}
    \[ 
        \Vert\phi_n(a)\Vert\longrightarrow \Vert a\Vert.
    \]
    Now, put \[\tau_n:=\tr_{k_n}\circ \phi_n: A_q(\mathbb C^r)\to \mathbb C,\]
    where $\tr_{k_n}$ is the normalized trace on $M_{k_n}(\mathbb C)$. After passing to a subsequence, the state $\tau_n$ converge in the weak-* topology to a state $\tau$ on $A_q(\mathbb C^r)$.

    Then the state $\tau$ is tracial. Indeed, for $a,b\in A_q(\mathbb C^r) $, because \(\tr_{k_n}(\phi_n(a)\phi_n(b))=\tr_{k_n}(\phi_n(b)\phi_n(a))\), we have
    \[\vert \tau_n(ab)-\tau_n(ba)\vert\leq \Vert \phi_n(ab)-\phi_n(a)\phi_n(b)\Vert + \Vert \phi_n(ba)-\phi_n(b)\phi_n(a)\Vert.\]
    Therefore $\tau(ab)=\tau(ba)$.

    $A_q(\mathbb C^r) $ has a unique tracial state, namely the canonical vacuum trace $\tau_q$, which was proved recently in \cite[Theorem C]{AmrutamJekelWasilewski2026}. Hence $\tau=\tau_q$. The map $\phi_n$, together with \eqref{Ext1}, therefore show that $\tau_q$ is a quasidiagonal trace. In particular, it is an amenable trace.

    On the other hand, since $A_q(\mathbb C^r)$ is exact \cite{Kuzmin, KennedyNica}, the von Nuemann algebra generated by the GNS representation w.r.t. an amenable trace is injective; see, for example, \cite{BrownOzawa}. Therefore, $\pi_{q}(A_q(\mathbb C^r))''=\Gamma_q(\mathbb R^r)$ is injective. But by Nou's result (\cite[Theorem 2]{Nou2004}), $\Gamma_q(\mathbb R^r)$ is noninjective whenever $r\geq 2$, this leads to a contraction. Thus $A_q(\mathbb C^r),r\geq 2$ is not quasidiagonal, which completes the proof of this theorem.
\end{proof}
\begin{remark}
The assumption $r\geq2$ in
Theorem~\ref{thm:q-gaussian-ext} is necessary. When $r=1$, the
algebra is generated by one self-adjoint operator, and hence
\(
A_q(\mathbb C)=
C^*(s_1^{(q)})\cong
C(\mathrm{sp}(s_1^{(q)})).
\)
is a separable nuclear $C^*$-algebra. The Choi-Effros' theorem on liftings of ucp maps on nuclear $C^*$-algebras (\cite{ChEf76}) implies that $\mathrm{Ext}(A_q(\mathbb C))$ is a group.
\end{remark}

\end{document}